\documentclass[leqno,a4paper]{article}
\usepackage{amsmath,amsthm,amssymb}
\usepackage{bbm,a4wide}

\newcommand{\cD}{\mathcal{D}}
\newcommand{\cF}{\mathcal{F}}
\newcommand{\cL}{\mathcal{L}}

\newcommand{\V}{\mathbb V}

\newcommand{\R}{{\mathbb R}}

\newcommand{\K}{{\mathbb K}}

\newcommand{\fa}{{\mathfrak{a}}}
\newcommand{\na}{{\nabla}}

\newcommand{\pa}{{\partial}}
\newcommand{\Om}{{\Omega}}
\newcommand{\Th}{{\Theta}}
\newcommand{\oOm}{{\overline{\Omega}}}
\newcommand{\oov}{{\overline{v}}}

\newcommand{\om}{{\omega}}
\newcommand{\Si}{{\Sigma}}

\newcommand{\al}{{\alpha}}

\newcommand{\hs}{{\widehat{\sigma}}}

\newtheorem{definition}{Definition}[section]
\newtheorem{remark}{Remark}[section]
\newtheorem{theorem}{Theorem}[section]
\newtheorem{lemma}[theorem]{Lemma}
\newtheorem{proposition}{Proposition}[section]
\newtheorem{corollary}{Corollary}[section]
\newtheorem{problem}{Problem}[section]
\newtheorem{assumption}{Assumption}[section]
\numberwithin{equation}{section}
\def\vs1{\vspace{1ex}}
\def\ro{\rho}

\def\O{\Omega}
\def\pa{\partial}
\def\ep{\epsilon}
\def\sph{\sphericalangle}

\def\dy{\displaystyle}
\def\be{\begin{equation}}
\def\ba{\begin{array}}
\def\ea{\end{array}}
\def\ee{\end{equation}}

\begin{document}
\title{\bf\large On nonlinear potential theory, and regular boundary
points, for the $\,p-$Laplacian in $N$ space variables }
\author{ H.~Beir\~ao da Veiga \thanks{Department of Mathematics,
 Pisa University, Via Buonarroti 1/C, 56127 Pisa, Italy,
 email: bveiga@dma.unipi.it}}
%\date{}
\maketitle
\begin{abstract} We turn back to some pioneering results
concerning, in particular, nonlinear potential theory and
non-homogeneous boundary value problems for the so called
$p-$Laplacian operator. Unfortunately these results, obtained at the
very beginning of the seventies, were kept in the shade. We believe
that our proofs are still of interest, in particular due to their
extreme simplicity. Moreover, some contributions seem to improve the
results quoted in the current literature.
\end{abstract}

\vspace{0.2cm}


%\medskip

%\noindent \textbf{Mathematics~Subject~Classification~(2000):} 35B65,35J57.
%\part{Foreword}\label{foreword}

\vspace{0.2cm}

\section{Introduction}\label{introduction}
At the very beginning of the seventies we proved a set of results
concerning nonlinear potential theory related to the so-called
$\,p-$Laplace operator. Following \cite{ricmat}, here we use the
symbol $"\,t\,"$ instead of the nowadays more common $"\,p\,"$, to
denote the leading integrability exponent (see \eqref{zeroquatro}).
In the 1972 paper \cite{ricmat}(see also \cite{c3}) we considered,
in a non-linear setting, notions such as barriers, order
preservation, capacitary potentials, regular boundary points, and so
on. This contribution seems almost forgotten in the subsequent
literature. However we believe that its topicality and interest
still remains, or has even grown. In fact, the basic ideas on which
the theory is founded, was emphasized by the original simplicity of
the broad lines. In Part I, we turn back to the results published in
reference \cite{ricmat}. We keep the presentation as close as
possible to the original paper. However, addition of suitable
remarks, together with some changes in notation, may help the
reader. By the way, we warn the reader that \cite{ricmat} is full of
small misprints, luckily very easy to single out and correct. In
Part II we turn back to an unpublished proof of a result stated in
reference \cite{ricmat}(theorem \ref{teo-unpub} below), and to a
related result proved in reference \cite{b1} (theorem \ref{teo-bvc}
below), both concerning regularity of boundary points for
$\,p$-Laplacian equations. The contribution of \cite{b1} to this
last problem was to prove H\H older continuity of the solutions to
the obstacle problem in the lower dimension $\,N-\,1\,.$ Below we
merely prove the continuity of the above solutions,
since this weaker is sufficient here. %

\vspace{0.2cm}

The main object of this work is the Dirichlet boundary value problem
\eqref{zeroseis}, whose prototype is the following problem
\begin{equation}\label{zerozero}\left\{
\begin{array}{ll}\vspace{1ex}
div \,\big(\,|\,\na \,u\,|^{t-\,2}\,\na\,u\,\big) =\,0 \ \mbox{ in }
\O\,,
\\%
u=\,\phi\ \mbox{ on } \partial \O\,.
\end{array}\right .
\end{equation}
For $\,t=\,2\,$ we get the classical Laplace equation. It is worth
noting that the theory developed in references \cite{ricmat} and
\cite{b1} could have been extended to similar, but more general,
equations. However, at that time, we were only interested in the
basic picture. Regular boundary points for the above Dirichlet
problem is here the core subject, since it establishes at any time
the direction to follow to get to its resolution. In equation
\eqref{zeroseis}, arbitrary continuous boundary data $\,\phi\,$ are
allowed. This leads us to consider two distinct notions of
solutions, generalized and variational.

\vspace{0.2cm}

We recall that a boundary point $\,y\,$ is said to be regular if to
each continuous boundary data $\,\phi\,$ the corresponding solution
is continuous in $\,y\,.$ Theorem \ref{teo-A} below (called theorem
A, in reference \cite{ricmat}) states that a point $\,y\,$ is
regular if and only if there is at $\,y\,$ a system of non-linear
barriers, see definition \ref{defzeroum}. By appealing to this last
result, we prove the theorem \ref{teo-B} (called theorem B in
reference \cite{ricmat}), which establishes that a point
$\,y\in\,\pa\,\Om\,$ is regular if and only if the $\,t-$capacitary
potentials of the sets $\,E_{\rho}\,$ satisfy \eqref{zerodezasseis},
for each positive real $\,m\,$, and each sufficiently small radius
$\,\rho\,,$ where
$$
\,E_{\rho}=\,(\complement \,\Om)(y,\,\rho)\,,
$$
denotes the complementary set of $\,\Om\,$ with respect to the
closed ball $\,\overline{I(y,\,\rho)}\,$.%

\vspace{0.2cm}

In part II, by appealing to the theorem \ref{teo-B}, we establish
two explicit, geometrical, sufficient condition for regularity. Let
us briefly illustrate these results.\par%
Denote by
\begin{equation}\label{densmedida}
\sigma(\rho)=\,\frac{|\,E_\rho\,|}{|\,I(y,\,\rho)\,|}%
\end{equation}
the density (with respect to the $N$-dimensional Lebesgue measure)
of $\,E_\rho\,$ with respect to the sphere $\,I(y,\,\rho)\,$. In
theorem \ref{teo-unpub} it is stated that there is a positive
constant $\,\Lambda\,$ such that if
\begin{equation}\label{densasr}
\big[\,\sigma(\rho)\,\big]^{\frac{t}{t-\,1}}\geq\,\Lambda\,(\log\,\log\,\rho^{-\,1}\,)^{-1}\,,
\end{equation}
for small, positive, values of $\,\rho\,,$ then the boundary point
$\,y\,$ is regular. Note that
\begin{equation}\label{melhor}
\lim_{\rho \rightarrow \,0} \,\sigma(\rho) =\,0
\end{equation}
is included, so the above condition is stronger than the usual
$N-$dimensional, external, cone property, and similar notions. This
result was already stated in the introduction of reference
\cite{ricmat} (due to a misprint, the second exponent $\,-\,1\,$ in
\eqref{densasr} was overlooked). At that time we did not publish the
proof, since we had used similar ideas in reference \cite{b1}, where
it was proved (still, appealing to theorem \ref{teo-B}) that a
boundary point $\,y\,$ is regular if a $\,(N-\,1)-$dimensional
external cone property is satisfied at the point $\,y\,$ (a
Lipschitz image of such a cone being sufficient).
See theorem \ref{teo-bvc} below.\par%
In fact, theorems \ref{teo-unpub} and \ref{teo-bvc} are corollaries
of the same result, theorem \ref{teo-nopub}, where it is proved that
the necessary and sufficient condition for regularity stated in
theorem \ref{teo-B} holds under the assumption \eqref{e62b}. The
proofs of theorems \ref{teo-nopub}, \ref{teo-unpub} and
\ref{teo-bvc} are shown in Part II below.

\vspace{0.2cm}

Our proofs do not require knowledge of particularly specialized
results. They appeal, in particular, to a suitable extension of De
Giorgi's truncation method to non-linear variational inequalities
with obstacles, following in particular reference \cite{c3} (se also
\cite{c1}). De Giorgi's truncation method was also used in reference
\cite{ziemer} to obtain the following sufficient condition for
regularity:
\begin{equation}\label{ziemer}
\limsup_{\rho \rightarrow \,0}
\,\textrm{cap}(E_{\rho}\,)\,\rho^{t-\,N}
>\,0\,,
\end{equation}
where $\,\textrm{cap}\equiv\, \textrm{cap}_{\,t}\,$ denotes (here
and in the sequel) the capacity of order $\,t\,.$ Since
$$
|\,E\,|^{\frac{N-\,t}{N}} \leq\, C\,\,\textrm{cap}_{\,t}\,E\,,
$$
condition \eqref{ziemer} leads to
\begin{equation}\label{ziemer2}
\limsup_{\rho \rightarrow \,0}
\,\frac{|\,E_\rho\,|}{|\,I(y,\,\rho)\,|} >\,0\,,
\end{equation}
which, basically, is equivalent to the $\,N-$dimensional external
cone property, as well as the corkscrew condition, stated in
\cite{heinonen}, theorem 6.31. This treatise furnishes a
wide-ranging excursion into the above and related results. See, in
particular, chapter 9.

\vspace{0.2cm}

Readers interested in a quick overlook on the main results should go
directly to definition \ref{defzeroum} and theorem \ref{teo-A}; To
definitions \ref{defzerodois} and \ref{demac}, and theorem
\ref{teo-B}; And, in Part II, to theorems \ref{teo-unpub} and
\ref{teo-bvc}.
\part{}
\section{Some definitions and main results}\label{defas}
We are concerned with the differential operator
\begin{equation}\label{zeroum}
\cL\, u =:\,div \,A(\nabla\,u)\,,
\end{equation}
where $\,A(p)\,$ denotes a continuous map from $\R^N$ into itself,
$u$ is a real function defined on an open subset of $\,\R^N\,,$ and
$\,\nabla\,u$ is its gradient. We assume the following conditions on
$\,A(p)\,$:
\begin{equation}\label{zerodois}
A(0)=\,0\,,
\end{equation}
\begin{equation}\label{zerotres}
\big(\,A(p)-\,A(q)\,\big)\cdot\,(p-\,q) > \,0\,,\quad \textrm{if}
\quad p \neq\,q\,,
\end{equation}
\begin{equation}\label{zeroquatro}
A(p)\cdot\,p \geq\,a\,|\,p\,|^t,\quad \textrm{if} \quad
|p|\geq\,p_{\,0}\,,
\end{equation}
\begin{equation}\label{zerocinco}
|A(p)|\leq\,a^{-\,1}\, |p\,|^{t-\,1},\quad \textrm{if} \quad
|p|\geq\,p_{\,0}\,,
\end{equation}
where $\,a>\,0$, $\,p_{\,0} \geq \,0\,,$ and $\,t>\,0$ are
constants. Further, $\,|\,x\,|\,$ and $\,x\cdot\,y\,$ denote,
respectively, the norm and the scalar product in $\,\R^N\,.$ Note
that the above assumptions imply $\,A(p)\cdot\,p >\,0\,,$ for all
$\,p \in\,\R^N\,.$\par%

\vspace{0.2cm}

In the following, $\,\Om\,$ is an open bounded subset of $\R^N$,
with boundary denoted by $\,\pa\,\Om\,$. We define
$H^{1,\,t}(\Om)\,$ as the completion of $\,C^1(\overline{\Omega})$
(or equivalently, $\,Lip\,(\overline{\Omega})\,)\,$ with respect to
the norm $\,\|\,v\,\|_{1,\,t}=\,\|\,v\,\|_t +\,\|\nabla\,v\,\|_t\,$.
$\,C^1(\overline{\Omega})$ is the set of functions which belong to
$\,C^0(\overline{\Omega})\,,$ and have continuous first order
partial derivatives in $\,\Om\,,$ which can be extended continuously
to $\,\overline{\Omega}\,.$ Furthermore, $\,H^{1,\,t}_0(\Om)\,$
denotes the closure in $\,H^{1,\,t}(\Om)\,$ of
$\,C^1_0(\overline{\Omega})$, the set of the
$\,C^1(\overline{\Omega})$ functions, with compact support in
$\,\Om\,$. See, for instance \cite{b9}. Furthermore,
$\,H^{1,\,t}_{loc}(\Om)\,$ denotes the set consisting of functions
defined in $\,\Om\,$, whose restriction to any $\,\Om'
\subset\,\subset \,\Om\,$ belongs to $\,H^{1,\,t}(\Om')\,.$\par%
We recall here the following property. Let $\,\phi(t)\,$ be a real,
Lipschitz continuous function of the real variable $\,t\,$, with, at
most, a finite number of points of non-differentiability. Further,
let $\,v \in\,H^{1,\,t}(\Om)\,.$ Then
$\,\phi(v(x))\in\,H^{1,\,t}(\Om)\,,$ moreover $
\pa_i\,\phi(v(x))=\,\phi'(v(x))\,\pa_i\,v(x)\,, $ a.e. in $\,\Om\,.$
In particular
\be\label{rq}%
\pa_i\,max\{v(x),\,k\}=\, \left\{\begin{array}{ll}\dy \pa_i\,v(x)
&\dy \mbox{ if }\ v(x)\geq\,k\,,\\
\hskip1cm 0 & \dy \mbox{ if }\ v(x)\leq\,k\,,%
\end{array}\right.\ee
a.e. in $\,\Om\,.$\par%
For convenience, we set
$$
\V=\,\V(\Om)=\,H^{1,\,t}(\Om)\,,  \quad
\V_0=\,\V_0(\Om)=\,H^{1,\,t}_0(\Om)\,,
$$
and so on.\par%
In the sequel we are interested in the Dirichlet problem
\begin{equation}\label{zeroseis}\left\{
\begin{array}{ll}\vspace{1ex}
\cL\,u =\,0 \ \mbox{ in } \O\,,
\\%
u=\,\phi\ \mbox{ on } \partial \O\,,
\end{array}\right .
\end{equation}
where $\,\cL\,u\,$ is defined by \eqref{zeroum}, and $\,\phi \in
\,C^0(\pa\,\Om)\,.$ In the sequel we show that to each $\,\phi \in
\,C^0(\pa\,\Om)\,$ there corresponds a unique solution $\,u
\in\,H^{1,\,t}_{loc}(\Om)\cap \,C^0(\Om)\,$ to the problem
\eqref{zeroseis}, see Theorem \ref{existesum}. This solution will be
called generalized solution.\par%
Since $\,A(p)\,$ may be merely continuous, local solutions of
problem $\,\cL\,u =\,0\,$ in $\, \O\,$ are understood in the
following, well known, weak sense. One considers the form
\begin{equation}\label{zerooito}
\fa(v,\,\psi)=:\,\int_{\Om }\, A(\na\,v\,)\cdot\,\na\,\psi\, dx\,,
\end{equation}
defined on $\,\V\times\,\V\,$, or on
$\,H^{1,\,t}_{loc}(\Om)\times\,D(\Om)\,,$ and give the following
definition.
\begin{definition}\label{nig}
We say that a function $\,u\,$ is a weak solution in $\,\Om\,$ of
problem
\begin{equation}\label{doisum} \cL \,u \equiv \,div
\,A(\na\,u)=\,0
\end{equation}
if $\,u \,$ belongs to $\, H^{1,\,t}_{loc}(\Om)\,$ and satisfies the
condition
\begin{equation}\label{doisdois}
\fa(u,\,\psi) =\,0\,, \quad \forall \,\psi\,\in \,\cD(\Om)\,.
\end{equation}
\end{definition}
Note that it immediately follows that \eqref{doisdois} holds for all
$\,\psi \in\, H^{1,\,t}(\Om)\,$ with
compact support in $\,\Om\,.$\par%
The above definition does not take into account boundary values. The
definition of generalized solution to the boundary value problem
\eqref{zeroseis}, where $\phi \in\,\,C^0(\pa\,\Om)\,,$ is given
below, see definition \ref{saneg}.  Generalized solutions to the
boundary value problem are defined as limits of suitable sequences
of variational solutions. In reference \cite{ricmat} we have used in
both cases the term "solution". However, for clarity, we decided to
use in these notes the two notions, "variational" and "generalized",
to denote related but distinct concepts.\par%
Next, we recall the definition of variational solution. Let $\,\phi
\in \,\V(\Om)\,.$ We set
\begin{equation}\label{zerodez}
\V_{\phi}(\Om)=\,\big\{\,v \in\,\V(\Om)\,:\,v-\,\phi \in
\,\V_0(\Om)\,\big\}\,.
\end{equation}
Properties (i) to (iv) below are easily shown.%

\vspace{0.2cm}

i) $\quad \fa(v,\,v-\,u) -\, \fa(u,\,v-\,u)\geq\,0\,,$ for all pair
$u,\,v \in\, \V(\Om)\,$ (monotonicity);\par%
(ii) $\quad \fa(u+\,t\,v,\,w)$ is a continuous function of the real
variable $\,t\,$, for all triad $\,u,\,v,\,w \in\,\V(\Om)\,$
(emicontinuity);\par%
(iii) $\quad \fa(v,\,v-\,u) -\, \fa(u,\,v-\,u)=\,0\,$ implies
$\,\nabla\,u=\,\nabla\,v\,$ in $\,\Om\,$; Moreover, if $\,u-\,v
\in\,\V_0(\Om)\,$ then $\,u=\,v\,$;\par%
(iv) One has (coercivity)
\begin{equation}\label{zerodoze}
\lim_{\|\,v\,\|_{1,\,t} \rightarrow \,\infty}\,
\frac{\fa(v,\,v)}{\|\,v\,\|_{1,\,t}}=\,+\,\infty\,,
\end{equation}
where $\,v\in\,\V_{\phi}(\Om)\,$.\par%
Existence and uniqueness of the solution to the following
variational problem is well known:
\begin{equation}\label{zerotreze}
u_1 \in\,\V_{\phi}(\Om)\,, \quad \fa(u_1,\,v)=\,0 \quad \forall
v\in\,\V_0(\Om)\,.
\end{equation}
Clearly, these solutions are weak solutions of \eqref{doisum}
in $\,\Om\,$. All this was already classical in the sixties.\par%
\begin{definition}\label{vasol}
The function $\,u=\,u_1 \,$ in \eqref{zerotreze} is, by definition,
the \emph{variational solution} to the Dirichlet problem
\eqref{zeroseis} when the boundary data is defined by means of an
element $\,\phi \in \,H^{1,\,t}(\Om)\,.$ In this case,
$\,u=\,\phi\,$ on $\,\partial \O\,$ means that $\,u-\,\phi \in
\,H^{1,\,t}_0(\Om)\,.$
\end{definition}
In the sequel, our first step is to extend to all continuous
boundary data $\,\phi\,$ the notion of solution. This will be done
as in reference \cite{ricmat}. Given $\,\phi \in\,\,C^0(\pa\,\Om)\,$
we consider an arbitrary sequence of functions in $\,\phi_n \in
\,C^1(\oOm)\,,$ which converge uniformly to $\,\phi\,$ on
$\,\pa\,\Om\,$, and we consider the sequence $\,u_n(x)\,$ consisting
of the variational solutions to the Dirichlet problem
\eqref{zeroseis}, with boundary data $\,\phi_n\,.$ Then we prove
(theorem \ref{teodoisquatro}) that the sequence $\,u_n(x)\,$
converges uniformly in $\,\Om\,$ to a function $\,u(x) \in
\,H^{1,\,t}_{loc}(\Om) \cap\,C^0(\,\Om)\,.$ Moreover, we show that
$\,u(x)\,$ is a weak solution in $\,\Om\,$ of problem
\eqref{doisum}, and also that it does not depend on the particular
sequence $\,\phi_n\,.$ So, to each continuous boundary data
$\,\phi\,$ there corresponds a unique element $\,u(x) \in
\,H^{1,\,t}_{loc}(\Om) \cap\,C^0(\,\Om)\,,$ obtained by the above
procedure. The above argument leads to the following, natural, definition.\par%
\begin{definition}\label{saneg}
Let $\,\phi \in\,\,C^0(\pa\,\Om)\,$ be given. By definition, the
above, unique, element $\,u(x) \in \,H^{1,\,t}_{loc}(\Om)
\cap\,C^0(\,\Om)\,$ is the generalized solution to the Dirichlet
problem \eqref{zeroseis} with the continuous boundary data
$\,\phi\,.$
\end{definition}
We anticipate the following result.
\begin{theorem}\label{existesum}
To each boundary value $\,\phi \in\,\,C^0(\pa\,\Om)\,$ there
corresponds a unique generalized solution to the Dirichlet problem
\eqref{zeroseis}.
\end{theorem}
It is worth noting that the auxiliary variational solutions
$\,u_n(x)\,$ used above are not necessarily continuous up to the
boundary, even though $\,\phi_n \in \,C^1(\oOm)\,.$ Even more, this
negative situation holds for generalized solutions. Hence, a crucial
problem is to study the possible continuity up to a boundary point
$\,y\,$ of the solutions to the Dirichlet problem. In this direction
we give the following definition.
\begin{definition}\label{regpoint}
We say that a point $\,y \in \,\pa\,\Om\,$ is regular, with respect
to $\,\Om$ and $\,\cL\,,$ if given an arbitrary data $\phi
\in\,\,C^0(\pa\,\Om)\,,$ the corresponding generalized solution $u$
of Dirichlet problem \eqref{zeroseis} satisfies the condition
\begin{equation}\label{zerosete}%
\lim_{x \,\in \Om\,, \,x\rightarrow \,y}\, u(x)=\,\phi(y)\,.%
\end{equation}
\end{definition}
As proved in the theorem \ref{teoquatrodois} below, the notion of
regular point has a local character.\par%
We remark that in the above definition, as in the following, we do
not assume (in any sense) that the continuous boundary data
$\,\phi\,$ is the trace on $\,\pa\,\Om\,$ of an element of
$\,H^{1,\,t}(\Om)\,.$%

\vspace{0.2cm}

For the Laplace operator, $\,A(p)=\,p\,,$ regular points have been
characterized by Wiener; see \cite{b14}, \cite{b15}, and Frostman
\cite{b5}. For linear operators with discontinuous coefficients,
$$
A_i(p)=\,\sum_{j} a_{i,\,j}(x)\,p_j\,,
$$
where
$$
\sum_{j} a_{i,\,j}(x)\,\xi_i\,\xi_j \geq\,\nu\,|\,\xi\,|^2\,,
$$
$\nu>\,0\,$, and $\,a_{i,\,j} \in\,L^{\infty}(\Om)\,,
i,\,j=\,1,...,N\,,$ such a characterization was given by Littmann,
Stampacchia, and Weinberger in \cite{b9}.%

\vspace{0.2cm}

The following definitions are crucial to the theory
(see \cite{b12}, definition 1.1, and remarks).%.
\begin{definition}\label{defourier}
Let $\,\Si\,$ be an open, bounded, set and $\,E \subset \overline
\Si\,$ be a measurable set. We say that $\,v \in\,H^{1,\,t}(\Si)\,$
satisfies the inequality $\,v \geq\,0\,$ on $\,E\,$ in the
$H^{1,\,t}(\Si)\,$ sense if there is a sequence $\,v_n \in
C^1(\overline \Si)\,$, convergent to $\,v\,$ in $H^{1,\,t}(\Si)\,$,
and satisfying $\,v_n \geq \,0\,$ on $\,E\,.$ Similarly, we define
$\,v \leq\,0\,$ on $\,E\,$, in the $H^{1,\,t}(\Si)\,$ sense.
Further, $\,v=\,0\,$ on $\,E\,$ if, simultaneously, $\,v \geq\,0\,$
and  $\,v \leq\,0\,$. Finally, $\,v\geq\,w\,$ on $\,E\,$, in the
$H^{1,\,t}(\Si)\,$ sense, if
$\,v-\,w\geq\,0\,$ on $\,E\,$, and so on.\par%
Furthermore, we denote respectively by $\,\sup_{E} \,v$ and
$\,\inf_{E} \,v$ the upper bound and the lower bound of $\,v\,$ on
$\,E\,$ in the $\,H^{1,\,t}(\Si)\,$ sense. Essential upper bounds
and lower bounds (i.e., up to sets of zero Lebesgue measure) are
denoted by the symbols  $\,\textrm{Sup}_{E} \,v$  and
$\,\textrm{Inf}_{E} \,v\,$, respectively.
\end{definition}

It is worth noting that the above definition is meaningless if the
$\,(N-\,2)-$ dimensional measure of set $\,E\,$ vanishes. This claim
is in general not true if we replace $\,N-\,2\,$ by $\,N-\,1\,.$ Let
us consider the following specific example, related to our results.
Assume that $\,E\,$ is an $\,(N-\,1)-$dimensional truncated cone
(see, for instance \eqref{cumum}) contained in a given sphere
$\,\Si\,$. Since elements $\,v \in\,H^{1,\,t}(\Si)\,$ do have a
trace (for instance, in the usual Sobolev's spaces sense) on the
surface $\,E\,$, it follows that if $\,v \in\,H^{1,\,t}(\Si)\cap
\,C^0(\Si)\,$ satisfies $\,v \geq\,m>\,0\,$ on $\,E\,,$ in the
$H^{1,\,t}(\Om)\,$ sense, then $\,v \geq\,m\,$ pointwisely on
$\,E\,$. However, if $\,E\,$ is an $\,N-\,2\,$ dimensional cone and
$\,t <\,N\,$, the result is not true in general. For instance the
continuous, constant, function $\,v=\,0\,$ in $\,\Si\,$ satisfies
$\,v \geq\,m>\,0,$ on $\,E\,$, in the sense of definition \ref{defourier}.\par%

\vspace{0.2cm}

To illustrate the results obtained in this work, we need additional
definitions and results. Given $y\in\,R^N$ and $\,\rho>\,0\,$, we
denote by $\,I(y,\,\rho)\,$ the open sphere with center in $y\,$ and
radius $\,\rho\,.$ If $\,B \subset\,\R^N\,,$ we set
$\,B(y,\,\rho)=\,B\,\cap\,I(y,\,\rho)\,.$ By  $\,\complement \,B\,$
and $\,\overline B\,$ we denote the complementary set and the
closure of $\,B\,$ in $\,R^N\,,$ respectively.

\vspace{0.2cm}

As in \cite{ricmat}, we give the following definitions.
\begin{definition}\label{emcima}
We say that $\,v \in \,H^{1,\,t}_{loc}(\Om)\,$ is a
\emph{supersolution} [resp., a \emph{subsolution}] in $\,\Om\,$,
with respect to the operator $\,\cL$, if
\begin{equation}\label{zerooito}
\fa(v,\,\psi)\geq\,0\,, \quad \forall \,\psi \in \cD(\Om)\,, \quad
\psi\geq\,0 \quad [\textrm{resp}.\, \psi\leq\,0\,]\,.
\end{equation}
\end{definition}
Obviously, if $\,v \in\,\V=\,H^{1,\,t}(\Om)\,$, then $\,\cD(\Om)\,$
may be replaced by $\,\V_0\,.$ Formally, a supersolution satisfies
$\,\cL \,v \leq\,0\,$ in $\,\Om\,.$\par%
The following definition generalizes Perron's notion of barrier (see
Perron \cite{perron} and Courant-Hilbert \cite{b2}, p.p. 306-312
and 341).%
\begin{definition}\label{defzeroum}
We say that there is a system of barriers at a point $\,y \in
\pa\,\Om\,$ with respect to $\cL\,$ if, given two positive arbitrary
reals $\,\rho\,$ and $\,m\,$, there exist a supersolution
$\,V\geq\,0\,$ and a subsolution $\,U\leq\,0\,$, which belong to
$\,\V \cap \,\,C^0(\Om)\,,$ and satisfy the following
conditions:%

\vspace{0.2cm}

$(j) \quad  V\geq\,m \quad \textrm{and} \quad \,U\leq\,-\,m\, \quad
\textrm{on} \quad (\pa\,\Om)\,\cap\,\complement \,I(y,\,\rho),$

\vspace{0.2cm}

$(jj) \quad \lim_{x \rightarrow\,y}\,V(x)=\,\lim_{x
\rightarrow\,y}\,U(x)=\,0\,.$
\end{definition}
In definition \ref{defzeroum}, and in the sequel, inequalities like
$\,V\geq\,m\,,$ $\,U\leq\,-\,m\,$, and so on, are to be intended in
the sense introduced in definition \ref{defourier}. Note that the
above definition does not change by restriction of the range of the
radius $\,\ro\,$ to values smaller than some positive
$\,\ro_0(y)\,$.\par%
Under suitable symmetry conditions, definition \ref{defzeroum} may
be simplified, as follows.
\begin{remark}\label{rem-1.1}
\rm{Define
\begin{equation}\label{umum}
B(p)=\,-\,A(-\-p)\,.
\end{equation}
The continuous function $\,B(p)\,$ inherits the properties
\eqref{zerodois},...,\,\eqref{zerocinco}. Furthermore, consider the
operator $\,\overline \cL \,w=\,div \,B(\na\,w)\,.$ The
transformation $\,w \rightarrow\, -\,w\,$ maps solutions of
\eqref{zerocatorze}, relative to one of the operators $\cL$ or
$\,\overline \cL\,,$ onto the solutions of \eqref{zeroquinze}
relative to the other operator, and reciprocally. Further, the same
transformation, maps supersolutions, solutions, and subsolutions,
relative to one of the operators onto, respectively, subsolutions,
solutions, and supersolutions, relative
to the other operator.\par%
In particular, if
\begin{equation}\label{umdois}
A(-\,p)=\,-\,A(p)\,,
\end{equation}
the transformation $\,w \rightarrow\, -\,w\,$ maps supersolutions
onto subsolutions, and reciprocally. In this case, it is sufficient
in definition \ref{defzeroum} to consider upper-solutions
$\,V\,$.\par%
Finally, if the function $\,A(p)\,$ is positively homogeneous
\begin{equation}\label{tressete}
A(s\,p)=\,s^{t-\,1}\,A(p), \quad \forall \,s>\,0\,,
\end{equation}
it is sufficient, in definition \ref{defzeroum}, to consider the
value $\,m=\,1\,.$}
\end{remark}
Main examples:
$\,A(p)=\,(\,1+\,|\,p\,|^2\,)^{\frac{(t-\,2)}{2}}\,p\,$ satisfies
\eqref{umdois}, and $\,A(p)=\,|\,p\,|^{t-\,2}\,p\,$ satisfies
\eqref{umdois}, and \eqref{tressete}. The differential equations
associate to this functions are the Euler equations to the extremals
of the integrals $\int \, (\,1+\,|\,\na\,u\,|^2\,)^{\frac{t}{2}}
\,dx\,$ and $\int \, |\,\na\,u\,|^t \,dx\,,$ respectively.

\vspace{0.2cm}

In section \ref{steoa} we prove the following result (see
\cite{ricmat}, theorem A):
\begin{theorem}\label{teo-A}
A point $\,y\,$ is regular if and only if there is at $\,y\,$ a
system of barriers.
\end{theorem}
As in reference \cite{ricmat}, the symbols $\,\Om\,$ and
$\,\Sigma\,$ denote suitable open sets. However, in this rewriting,
we make the reading easier by a better use of the above symbols.

\vspace{0.2cm}

Let $\,\Si\,$ be an open bounded set, $\,E \subset \Sigma\,$ be a
closed set, and $\,m\,$ be a positive constant (the fact that
$\,\Si\,$ is assumed to be a sphere is not necessary here). We
introduce the following convex, closed, subsets of
$\,\V_0(\Sigma)\,$.
\begin{equation}\label{zeroonze}
\K_m(\Sigma)=\,\big\{\,v \in\,\V_0(\Sigma)\,:\,v \geq\,m \quad
\mbox{on}\quad E\,\big\}\,,
\end{equation}
and
\begin{equation}\label{zeroonze2}
\K_{-\,m}(\Sigma)=\,-\,\K_m(\Sigma)=\,\big\{\,v
\in\,\V_0(\Sigma)\,:\,v \leq\,-\,m \quad \mbox{on}\quad E\,\big\}\,.
\end{equation}
Inequalities are in the $\,H^{1,\,t}(\Sigma)\,$ sense.\par%
Obviously, properties (i) to (iii) hold with $\,\Om\,$ replaced by
$\,\Si\,.$ Moreover, as easily shown, the coercivity property (iv)
holds by replacing $\,v\in\,\V_{\phi}(\Om)\,$ by
$\,v\in\,\K_m(\Si)\,,$ or by $\,v\in\,\K_{-\,m}(\Si)\,.$ Hence, from
properties (i) to (iv), together with well known general theorems
(see Hartman-Stampacchia \cite{b6} and J.-L.Lions \cite{b8}),
existence and uniqueness of solutions to the following
two problems follows.\par%
\begin{equation}\label{zerocatorze}
u_2 \in\,\K_m(\Si)\,, \quad \fa(u_2,\,v-\,u_2)\geq\,0\,, \quad
\forall v\in\,\K_m(\Si)\,;
\end{equation}
\begin{equation}\label{zeroquinze}
u_3 \in\,\K_{-\,m}(\Si)\,, \quad \fa(u_3,\,v-\,u_3)\geq\,0\,, \quad
\forall v\in\,\K_{-\,m}(\Si)\,.
\end{equation}

\vspace{0.2cm}

Next we introduce the $\,t-$capacitary potentials. The following
definition is related to the notion of capacity used by Serrin in
\cite{b11}.%
\begin{definition}\label{defzerodois}
Let $\,\Sigma\,$ be an open sphere, $\,E\subset \Sigma\,$ be a
closed set, and $\,m\,$ be a positive real. The solutions to the
problems \eqref{zerocatorze} and \eqref{zeroquinze} are called
$t-$capacitary potentials of the set $\,E\,$ with respect to the
non-linear operator $\,\cL\,$, the real $\,m\,$ and the sphere
$\,\Sigma\,$. Since $\,t\,$ is fixed, we drop the label $\,t\,$.
\end{definition}
In definition \ref{defzerodois}, the dependence on the particular
fixed sphere $\,\Si\,$ is without significance. In particular, the
numerical values of the related capacities remain equivalent
provided that the distances from the sets $\,E\,$ to the boundary
$\,\pa\,\Si\,$ have a positive, fixed, lower bound. From now on we
fix, once and for all, a sphere
$$\,\Si =\,I(y_0,\,2\,R)\,
$$
such that
$$
\Om \subset \,\,I(y_0,\,R)\,.
$$
So
$$
dist(\Om,\,\pa\,\Si) \leq\,R\,.
$$%
Further, for each couple $\,y,\,\rho\,,$ where $\,y \in\,\pa\,\Om
\,$ and $\, 0<\,\rho<\,\frac{R}{2} \,,$ we set
\begin{equation}\label{erro}
E_{\rho}=\,(\complement\,\Om)\cap\,\overline{I(y,\,\rho)}\,.
\end{equation}
\begin{definition}\label{demac}
We denote by $\,u_{m,\,\rho\,}\,$ and $\,u_{-m,\,\rho\,}\,$ the
capacitary potentials of the above sets $\,E_{\rho}\,$ relative to
the values $m$ and $-\,m$ respectively.%
\end{definition}

\vspace{0.2cm}

In section \ref{steob} we prove the following result (see
\cite{ricmat}, theorem B):
\begin{theorem}\label{teo-B}
A point $\,y\in\,\pa\,\Om\,$ is regular if and only if the
capacitary potentials of the sets $\,E_{\rho}\,$ are continuous in
$\,y\,$. More precisely, if and only if
\begin{equation}\label{zerodezasseis}\left\{
\begin{array}{ll}\vspace{1ex}
\lim_{x \rightarrow \,y} u_{m,\,\rho}(x)=\,m\,,
\\%
\lim_{x \rightarrow \,y} u_{-m,\,\rho}(x)=\,-\,m\,,
\end{array}\right.
\end{equation}
for each couple $\,\rho\,$, $\,m\,$ as above (or, equivalently, for
a sequence $\,(\rho_n,\,m_n)\,$ such that
$\,(\rho_n,\,m_n)\rightarrow\, (0,\,+\,\infty)\,.$
\end{theorem}
From theorem \ref{teo-B}, together with the immersion of
$\,H^{1,\,t}(\Si)\,$ in $\,C^{0,\,1-\,\frac{N}{t}}(\overline\Si)\,$,
one gets the following result.
\begin{corollary}\label{coro-C}
Any boundary point is regular with respect to the operator $\,\cL\,$
if  $\,t>\,N\,$.
\end{corollary}
\section{Maximum principles and related results}
In this section we state some results concerning maximum principles,
order preservation, and similar notions. Related results may be
found, for
instance, in \cite{b13}, \cite{b3}, and \cite{b10}.\par%
The section is divided into two subsections. The first one concerns
variational solutions in $\,\Om\,$ to the non-linear boundary value
problem \eqref{zeroseis}. The second one concerns solutions to the
variational inequalities \eqref{zerocatorze} and \eqref{zeroquinze},
which describe obstacle problems in $\,\Si\,$.
\subsection{Variational solutions in $\,\Om\,.$ }
We denote by $\,|B|\,$ the Lebesgue measure of a set $\,B\,.$  By
$\,c,\,c_0,\,c_1\,,$ etc., we denote positive constants that depend,
at most, on $\,t,\,N,\,a,\,,$ and $\,p_0\,.$ The same symbol may be
used to denote different constants of the same type.\par%
One has the following \emph{maximum principle}.
\begin{lemma}\label{lemumdois}
The (variational) solution $\,u=\, u_1\,$ of problem
\eqref{zerotreze} satisfies the estimates
\begin{equation}\label{umtres}
\inf_{\pa\,\Om} \phi \leq\,\textrm{Inf}_{\,\Om} \,u
\leq\,\textrm{Sup}_{\,\Om}\, u \leq\,\sup_{\pa\,\Om} \phi\,.
\end{equation}
\end{lemma}
\begin{proof}
We prove that $\,Sup_{\,\Om}\, u \leq\,k\,,$ where
$\,k=\,\sup_{\pa\,\Om} \phi\,.$ For convenience we set
$$
A(k)=\,\{x \in\,\Om:\,u(x) \geq\,k\,\}\,.
$$
If $|\,A(k)\,|=\,0\,$ the thesis is obvious. Assume that
$|\,A(k)\,|>\,0\,,$ and set $\,v=\,max \{u-\,k,\,0\}\,$. Since
$\,v\in \,H^{1,\,t}_0(\Om)\,,$ it follows from \eqref{zerotreze}
that
$$
\int_{A(k)} \,A(\na\,u)\cdot\,\na\,u \,dx=\,0\,.
$$
This equation, together with \eqref{zerodois} and \eqref{zerotres},
shows that $\,\na\,u=\,0\,$ on $\,A(k)\,$, so $\,u=\,k\,$ on
$\,A(k)\,$. This proves our thesis. A similar argument proves the
first inequality \eqref{umtres}.
\end{proof}
\begin{lemma}\label{lemumtres}
\emph{Order Preserving}: Let $w$ be a subsolution, $z$ a
supersolution, and assume that $\,w\leq\,z\,$ on $\,\pa\,\Om\,$.
Then $\,w(x)\leq\,z(x)\,$ almost everywhere in $\,\Om\,$.
\end{lemma}

\begin{proof}
Set $\,\eta=\,min\{0,\,z-\,w\}\,.$ It follows that $\,\eta\in
\,H^{1,\,t}_0(\Om)\,,$ moreover $\,\eta(x) \leq\,0\,.$ By taking
into account definition \ref{emcima}, we may write
\begin{equation}\label{dezass}
\int_{\Om} \,A(\na\,w)\cdot\,\na\,\eta \,dx=\, \int_{\{w\geq\,z\}}
\,A(\na\,w)\cdot\,\na\,(z-\,w)\,dx \geq\,0\,,
\end{equation}
and
\begin{equation}\label{dezaset}
\int_{\{w\geq\,z\}} \,A(\na\,z)\cdot\,\na\,(z-\,w)\,dx \geq\,0\,,
\end{equation}
where $\,\{w\geq\,z\}=\,\{x\in\,\Om :\,w(x) \geq\,z(x)\,\}\,.$  From
\eqref{dezass} and \eqref{dezaset} it follows that
\begin{equation}\label{dezota}
\int_{\{w\geq\,z\}} \,(\,A(\na\,z) -\,A(\na\,w)\,)
\cdot\,(\,\na\,z-\,\na\,w\,) \,dx \leq\,0\,.
\end{equation}
This inequality together with \eqref{umdois} imply
$\,\,\na\,(w-\,z\,)=\,0\,$ on  $\,\{w-\,z\,\geq\,0\}\,.$ By
appealing to the hypothesis $\,w-\,z\,\leq\,0\,$ on $\,\pa\,\Om\,,$
the thesis follows.
\end{proof}
\begin{corollary}\label{coroumquatro}
If $\,u\,$ and $\,v\,$ are two (variational) solutions, which belong
respectively to $\,\V_{\phi}\,$ and $\,\V_{\psi}\,,$ then
\begin{equation}\label{umquatro}
Sup_{\Om} \,|\,u-\,v\,| \leq\,\sup_{\pa\,\Om}\, |\,\phi-\,\psi\,|\,.
\end{equation}
\end{corollary}
\begin{proof}
Set $\,\eta=\,\sup_{\pa\,\Om}\, |\,\phi-\,\psi\,|\,.$ The function
$\,w=\,v+\,\eta\,$ is  a variational solution in $
H^{1,\,t}_{\psi+\,\eta}(\Om)\,,$ moreover $\,u \leq\,w\,$ on
$\,\pa\,\Om\,.$ By lemma \ref{lemumtres} it follows that $\,u
\leq\,w=\,v+\,\eta\,$ a.e. in $\,\Om\,,$ that is $\,u-\,v\,
\leq\,\eta\,$ a.e. in $\,\Om\,.\,$ Similarly, one proves that $\,
v-\,u\, \leq\,\eta\,,$ a.e. in  $\,\Om\,.$ These two relations yield
the thesis.
\end{proof}
\subsection{Variational inequalities in $\,\Si\,.$ }
In this subsection $\,\Si\,,$ $\,E\,$ and $\,m\,,$ are as in
definition \ref{defzerodois}.
\begin{lemma}\label{lemumcinco}
Let $\,u=\,u_2\,$ be the solution of problem \eqref{zerocatorze}.
Then $\,u(x)\leq\,m\,$ almost everywhere in $\,\Sigma\,$. In
particular, $\,u=\,m\,$ on $\,E\,$.\par%
Analogously, the solution $\,u=\,u_3\,$ of \eqref{zeroquinze}
satisfies the inequality $\,u(x)\geq\,-\,m\,$ almost everywhere in
$\,\Sigma\,$. In particular, $\,u=\,-m\,$ on $\,E\,$.
\end{lemma}
\begin{proof}
Let be $\,u=\,u_2\,$, and set $\,v=\,\min\{u,\,m\}\,.$ Since $\,v\in
\K_m(\Si)\,,$ from \eqref{zerocatorze} we get
$$
\int_{\Si} \,A(\na\,u)\cdot\,\na\,(v-\,u)\, dx \geq\,0\,,
$$
that is
\begin{equation}\label{dosdos}
\int_{B_m}\,A(\na\,u)\cdot\,\na\,u \,dx \leq\,0\,.
\end{equation}
From \eqref{dosdos}, \eqref{zerodois}, and \eqref{zerotres} it
follows that $\,\na\,u=\,0\,$ a.e. on the set $\,\{x \in\,\Si:\,u(x)
\geq\,m\,\}\,$. From this last property, since $\,u\in \K_m(\Si)\,$,
it readily follows that $\,u=\,m\,$ on $\,E\,.$ The second part of
the lemma may be proved in a similar way, or as a consequence of the
first part, together with the remark \ref{rem-1.1}.
\end{proof}
\begin{lemma}\label{lemumseis}
The solution $\,u=\,u_2\,$ of problem \eqref{zerocatorze} solves, in
$\,\Sigma-\,E\,$, the problem
\begin{equation}\label{umcinco}
\int_{\Sigma -\,E} \,A(\na\,u)\cdot\,\na\,v \,dx=\,0\,, \quad
\forall \,v\,\in H^{1,\,t}_0(\Sigma -\,E)\,.
\end{equation}
Moreover, $\,u_2\,$ is a super-solution in $\,\Sigma\,$. Similarly,
the solution $\,u=\,u_3\,$ of \eqref{zeroquinze} solves, in
$\,\Sigma-\,E\,$, the problem \eqref{umcinco}, and is a sub-solution
in $\,\Sigma\,$.
\end{lemma}
\begin{proof}
Equation \eqref{zerocatorze} may be written in the form%
\begin{equation}\label{doqas}%
\int_{\Si} \,A(\na\,u)\cdot\,\na\,(w-\,u)\, dx \geq\,0\,,\,, \quad
\forall w\in\,\K_m(\Si)\,.%
\end{equation}
Given $\,v\,\in H^{1,\,t}_0(\Sigma -\,E)\,$, denote by $\,\oov\,$
the function equal to $\,v\,$ in $\,\Si -\,E\,,$ and vanishing on
$\,E\,.$ By the construction, the functions $\,u+\,\oov\,$ and
$\,u-\,\oov\,$ belong to $\,\K_m(\Si)\,.$ By replacing these
functions in equation \eqref{doqas} we obtain \eqref{umcinco}.\par%
Furthermore, $\,u\,$ is a super-solution. In fact, let $\,\psi
\in\,C^{\infty}_0(\Om)\,,$ be non-negative. Then the function
$\,w=\,u+\,\psi\,$ belongs to $\,\K_m(\Si)\,.$ By using it as test
function in equation \eqref{doqas}, one proves \eqref{zerooito}.\par%
The second part of the lemma may be obtain similarly or,
alternatively, by appealing to the remark \ref{rem-1.1}.
\end{proof}
\section{A convergence result. Proof of the existence theorem \ref{existesum}}\label{sectres}
In this section we associate to each boundary data $\,\phi\,\in
\,C^0(\pa\,\Om)\,$ a weak solution $\,u\,$ in $\,\Om\,$ of equation
\eqref{doisum}. Recall that, by definition, $\,u\,$ is a weak
solution of \eqref{doisum} in $\,\Om\,$ if $\,u \in\,
H^{1,\,t}_{loc}(\Om)\,$ satisfies \eqref{doisdois}, namely
$$
\int_{\Om} \,A(\na\,u)\cdot\,\na\,\psi \,dx=\,0\,, \quad \forall
\,\psi\,\in \,\cD(\Om)\,.
$$
As already remarked, it immediately follows that \eqref{doisdois}
holds for all $\,\psi \in\, H^{1,\,t}(\Om)\,,$ with compact support
in $\,\Om\,.$
\begin{remark}\label{remdoisum}
$\,L^{\infty}(\Om)$ solutions to equation \eqref{doisdois}
necessarily belong to $\,C^0(\Om)\,.$
\end{remark}
In fact, the above solution is locally H\H older continuous in
$\,\Om\,$, see Ladyzhenskaya-Ural'tseva \cite{b7}. Actually,
continuity may be proved by appealing to a simplification of the
argument used in Part II below.
\begin{lemma}\label{lemdoisdois}
A family of solutions to equation \eqref{doisdois}, equi-bounded in
 $\,L^{\infty}(\Om)\,,$ is necessarily  equi-bounded in
 $\,H^{1,\,t}(\Om')\,,$ for each $\,\Om' \subset\,\subset \,\Om\,.$
\end{lemma}
\begin{proof}
From the properties of $\,A(p)\,$ it immediately follows that
\begin{equation}\label{doistres}\left\{
\begin{array}{ll}\vspace{1ex}
A(p) \cdot\,p \geq\,a\,|p|^t -\,a\,p_{\,0}^{\,t}\,,
\\%
|\,A(p)\,| \leq\,a^{-\,1}\,|p|^{t-\,1}+\,d_0\,,
\end{array}\right.
\end{equation}
where $\,d_0\,$ is a non-negative constant. Let be $\,k>\,0\,$, and
consider the family $\,\cF\,$ consisting of the solutions to
\eqref{doisdois} for which $\,Sup_{\Om}\,|\,u(x)\,| \leq\,k\,$.
Equi-boundedness of $\,\|\,u\,\|_{t,\,\Om}\,$ is obvious. let us
proof the equi-boundedness of $\,\|\,\na\,u\,\|_{t,\,\Om}\,.$ Let
$\,\Lambda\,$ be an open set such that $\,\Om' \subset\,\subset
\Lambda \,\subset\,\subset \,\Om\,$, and let $\,\phi\,$ be a regular
function, $\,0\leq\,\,\phi(x) \leq\,1\,$, equal to $\,1\,$ in
$\,\Om'\,$, and vanishing on $\,\Om -\,\Lambda\,$. One easily shows
that
\begin{equation}\label{doisquatro}
\int_{\Lambda} \,A(\na\,u)\cdot\,(\na\,u\,)\,\,\phi^t \,dx \leq\,t\,
\int_{\Lambda} \,|\,A(\na\,u)\,| \,|\,\na\,\phi\,|\,|\,u\,|
\,\,\phi^{t-\,1} \,dx\,.
\end{equation}
By appealing to H\H older's inequality one gets
$$
\int_{\Lambda} \,A(\na\,u)\cdot\,(\na\,u\,)\,\,\phi^t \,dx
\leq\,C\,\big(\, \int_{\Lambda}
\,|\,A(\na\,u)\,|^{\frac{t}{t-\,1}}\,\,\,\phi^t
\,dx\,\big)^{\frac{t-\,1}{t}}\,,
$$
where $\,C=\,t\,k\,\|\,\na\,\phi\,\|_{t,\,\Lambda}\,.$ The last inequality together with \eqref{doistres}
leads to
$$
a\,\int_{\Lambda} \,|\,\na\,u\,|^t\,\phi^t \,dx \leq\,C_0\,\big(\,
\int_{\Lambda} \,|\,\na\,u\,|^t\,\phi^t
\,dx\,\big)^{\frac{t-\,1}{t}}\,+\,C_1\,.
$$
Since $\,\frac{t-\,1}{t}<\,1\,,$ it readily follows that the
integral in the left hand side of the above inequality is bounded by
a constant $\,C_2\,.$ So
$$
\int_{\Om'} \,|\,\na\,u\,|^t \,dx \leq \,\int_{\Lambda}
\,|\,\na\,u\,|^t\,\phi^t \,dx \leq \,C_2\,.
$$
\end{proof}

\begin{lemma}\label{lemdoistres}
Let $\,\{u_n\}\,$ be a sequence of solutions to \eqref{doisdois}, equi-bounded in
 $\,L^{\infty}(\Om)\,,$ and uniformly convergent in $\,\Om\,$ to a function $\,u(x)\,.$ Then
 $\,u(x)\,$ is a solution to \eqref{doisdois}\,.
\end{lemma}
\begin{proof}
Note that $\,u_n\,\in C^0(\Om)\,,$ as follows from the remark
\ref{remdoisum}. Lemma \ref{lemdoisdois} shows that $\,u
\in\,H^{1,\,t}(\Om')\,,$ for each $\,\Om'\,$ as above. Let $\,u^0
\in\,H^{1,\,t}(\Om')\,$ be the variational solution in $\,\Om'\,$ of
the problem $\,\cL\,u^0=\,0\,$ in $\,\Om'\,,$ $\,u^0-\,u\,\in\,
H^{1,\,t}_0(\Om')\,.$ By applying the corollary \ref{coroumquatro}
to the functions $\,u^0\,$ and $\,u_n\,$ it follows that
\begin{equation}\label{doiscinco}
Sup_{\,\Om'}\,|\,u_n-\,u^0\,| \leq\,\sup_{\pa\,\Om'}\,|\,u_n-\,u\,|\,.
\end{equation}
Since $\,u_n(x) \rightarrow\,u(x)\,$ uniformly in $\,\Om\,$, from
\eqref{doiscinco} it follows that $\,u_n(x) \rightarrow\,u^0(x)\,$
uniformly in $\,\Om'\,.$ So, $\,u^0(x)=\,u(x)\,$ in  $\,\Om'\,.$ In
particular, $\,\cL\,u=\,0\,$ in $\,\Om'\,.$ From the arbitrarity of
$\,\Om'\,,$ the thesis follows (note that local uniform convergence
in $\Om\,$ would be sufficient here).%.
\end{proof}

\vspace{0.2cm}

The following statement corresponds to the Theorem 2.4 in reference
\cite{ricmat}.
\begin{theorem}\label{teodoisquatro}
To each $\,\phi\,\in C^0(\pa\,\Om)\,$ there corresponds a (unique)
function $\,u(x)\,$ such that the following holds:\par%
Let $\,\{\,\phi_n\,\}\,$ be an arbitrary sequence of functions in
$\,C^1(\oOm)\,$ uniformly convergent to $\,\phi\,$ on $\,\pa\,\Om\,$
(it is well know that these sequences exist). Further, denote by
$\,u_n(x)\,$ the variational solutions to the problem
$\,\cL\,u_n=\,0\,$, $\,u_n \in\,\V_{\phi_n}\,.$ Then the sequence
$\,\{\,u_n\}\,$ converges uniformly in $\,\Om\,$ to a function
$\,u(x)\,.$ Moreover, the function $\,u(x)\,$, which belongs to
$\,H^{1,\,t}_{loc}(\Om) \cap\,C^0(\,\Om)\,,$ is a weak solution in
$\,\Om\,$, i.e. $\,u\,$ solves \eqref{doisdois}.
\end{theorem}
\begin{proof}
Let $\,\phi\,$, $\,\phi_n\,,$ and $\,\,u_n\,$ be as in the above
statement. The variational solutions $\,\,u_n\,$ are continuous in
$\,\O\,$, see the remark \ref{remdoisum}. Clearly, they are also
equi-bounded. By corollary \ref{coroumquatro} it follows that, for
all couple of indexes $\,m,\,n\,,$
\begin{equation}\label{umquatro}
Sup_{\,\Om} \,|\,u_n-\,u_m\,| \leq\,\sup_{\pa\,\Om}\,
|\,\phi_n-\,\phi_m\,|\,.
\end{equation}
So the sequence $\,\{\,u_n(x)\}\,$ is uniformly convergent in
$\,\O\,$ to some $\,u(x)\,\in C^0(\Om)\,.$ Lemma \ref{lemdoistres}
shows that $\,u(x)\,$ is a weak solution in $\,\O\,.$ Moreover, by
appealing to lemma \ref{lemdoisdois}, we get $\,u
\in\,H^{1,\,t}_{loc}(\Om)\,.$ Furthermore, the limit $\,u\,$ is
independent of the particular sequence $\,\{\,\phi_n\}\,$, as
follows from \eqref{umquatro} applied to two distinct, arbitrary,
sequences $\,(\phi_n,\,u_n)\,$ and $\,(\psi_n,\,v_n)\,$. This
argument also proves the uniqueness of the solution $\,u\,.$
\end{proof}
The theorem \ref{teodoisquatro} justifies the definition \ref{saneg}
of generalized solution given in section \ref{defas}, and also
proves the existence and uniqueness Theorem \ref{existesum}.

\vspace{0.2cm}

It is worth noting that from definition \ref{saneg}, lemma
\ref{lemumdois}, and corollary \ref{coroumquatro}, it follows that
if $\,u\,$ and $\,v\,$ are the solutions corresponding to the
continuous data $\,\phi\,$ and $\,\psi\,$, then
$$
\min_{\pa\,\Om} \phi \leq\,Inf_{\,\Om}\, u \leq\,Sup_{ \,\Om} \,u
\leq\,\max_{\pa\,\Om} \phi\,,
$$
and
$$
Sup_{\,\Om} \,|\,u -\,v\,| \leq\,\max_{\pa\,\Om} |\,\phi-\,\psi\,|\,.
$$
Minimum and maximum are used here in the very classical sense.

\section{Proof of Theorem \ref{teo-A}}\label{steoa}
In this section we prove the theorem \ref{teo-A}. We denote by
$\,C^1(\pa\,\Om)\,$ the functional space consisting on the
restrictions to $\,\pa\,\Om\,$ of functions in $\,C^1(\oOm)\,.$
\begin{lemma}\label{lemtresum}
A point $\,y \in\,\pa\,\Om\,$ is regular if and only if condition
\eqref{zerosete} holds for each $\,\phi\in \,C^1(\pa\,\Om)\,.$
\end{lemma}
\begin{proof}
Let $\,u\,$ be the solution corresponding to a given data $\,\phi\in
\,C^0(\pa\,\Om)\,,$ and let $\,\{\,\phi_n\,\}\,$ and
$\,\{\,u_n\,\}\,$ be as in theorem \ref{teodoisquatro} (by the way,
note that the solutions $\,\,u_n\,$ are variational and
generalized). Define $\,\overline{u}_n(x)\,$ by
$\,\overline{u}_n(x)=\,u_n(x)\,$ in $\,\Om\,$,
$\,\overline{u}_n(y)=\,\phi_n(y)\,,$ and define
$\,\overline{u}(x)\,$ by $\,\overline{u}(x)=\,u(x)\,$ in $\,\Om\,$,
$\,\overline{u}(y)=\,\phi(y)\,.$ The functions
$\,\overline{u}_n(x)\,$ are, by the assumptions, continuous in
$\,\Om \cup\,\{y\}\,$, and uniformly convergent in $\,\Om
\cup\,\{y\}\,$ to the function $\,\overline{u}(x)\,.$ So,
$\,\overline{u}(x)\,$ is continuous in $\,\Om \cup\,\{y\}\,$.
\end{proof}
\begin{proof} \emph{of theorem} \ref{teo-A}.\par%
Necessary condition: assume that $\,y \in\,\pa\,\Om\,$ is regular.
Given $\,\rho\,$ and $\,m\,$, consider the restriction to
$\,\pa\,\Om\,$ of the function
$\,h(x)=\,m\,\frac{|\,x-\,y|^2}{\rho^2}\,.$ This function belongs to
$\,C^1(\pa\,\Om)\,.$ Let $\,V(x)\,$ be the solution with $\,h(x)\,$
as boundary data. By the construction, $\,V(x)\,$ satisfies
condition (j) in definition \ref{defzeroum}. Further, by the
definition of regular point, $\,V(x)\,$ satisfies the condition
(jj). Similarly, by considering the data $\,-\,h(x)\,,$ one proves
the existence of the function $\,U(x)\,$ required in  definition \ref{defzeroum}.\par%

Sufficient condition: Assume that, at point $\,y\,,$ there exists a
system of barriers. By lemma \ref{lemtresum} we may assume that
$\,k(x) \in \,C^1(\pa\,\Om)\,$. Let $\,u(x)\,$ be the corresponding
solution and set
\begin{equation}\label{tresum}
M =\,\sup_{\pa\,\Om}\,|\,k(x)\,|\,.
\end{equation}
Given $\,\ep \,>0\,$, there is $\,\rho_{\ep}>\,0\,$ such that
\begin{equation}\label{tresdois}
|\,k(x)-\,k(y)\,| <\,\frac{\ep}{2} \quad \textrm{if} \quad
|\,x-\,y\,| <\,\rho_{\ep}\,, \quad x \in\,\pa\,\Om\,.
\end{equation}
Let $\,V\,$ and $\,U\,$ be barriers related to the values
$\,\rho=\,\rho_{\ep}\,$ and $\,m=\,M\,$. Then (see also \cite{b9})

\begin{equation}\label{trestres}
V(x) \geq\,M\,, \quad \textrm{and} \quad U(x) \leq\,-\,M\,, \quad
\textrm{on} \quad (\pa\,\Om) \cap\,\complement \,I(y,\,\rho)\,.
\end{equation}
By appealing to \eqref{tresum}, \eqref{tresdois}, and
\eqref{trestres}, we show that
\begin{equation}\label{tresquatro}\left\{
\begin{array}{ll}\vspace{1ex}
k \leq\,k(y) +\,\frac{\ep}{2}+\,V\,,
\\%
k \geq\,k(y) -\,\frac{\ep}{2}+\,U\,,
\end{array}\right.
\end{equation}
on $\pa\,\Om\,$.\par%
From \eqref{tresquatro} and lemma \ref{lemumtres} it follows that
\begin{equation}\label{trescinco}\left\{
\begin{array}{ll}\vspace{1ex}
u(x) \leq\,k(y) +\,\frac{\ep}{2}+\,V(x)\,,
\\%
u(x) \geq\,k(y) -\,\frac{\ep}{2}+\,U(x)\,,
\end{array}\right.
\end{equation}
almost everywhere in $\,\Om\,$, since $\,k(y)
+\,\frac{\ep}{2}+\,V(x)\,$ is a super-solution, etc. Furthermore,
the property (jj) in definition \ref{defzeroum} implies the
existence of $\,\overline{\rho}_{\ep}>\,0 \,$ such that
\begin{equation}\label{tresseis}
x \in\,\Om\, \cap\,I(y,\,\overline{\rho}_{\ep})\quad \Longrightarrow
\quad |\,V(x)\,| \leq\,\frac{\ep}{2} \quad \textrm{and} \quad
|\,U(x)\,| \leq\,\frac{\ep}{2}\,.
\end{equation}
From \eqref{trescinco} and \eqref{tresseis} we show that
$$
\,x \in\,\Om\, \cap\,I(y,\,\overline{\rho}_{\ep})\,\Longrightarrow
\,-\,\ep \leq\,u(x)-\,h(y)\leq\,\ep\,.
$$
So, $\,\lim_{x \rightarrow\,y}\,u(x)=\,h(y)\,.$ Hence $\,y\,$ is
regular.
\end{proof}

\section{Proof of Theorem \ref{teo-B}}\label{steob}
We start by some preliminary results.
\begin{lemma}\label{lemquatroum}
The Lipschitz continuous function
$$
u(x)=\,\,\al \,|\,x-\,y\,| +\, \beta\,,
$$
where $\al$ and $\beta$ are constants, is a super-solution if
$\al>\,0$, and a sub-solution if $\al<\,0$.
\end{lemma}
\begin{proof}
Without loss of generality we assume that $\,u(x)=\,\al\,r\,$, where
$\,r=\,|\,x\,|\,.$ We start by assuming that $\,A(p)\,$ is
indefinitely differentiable. By taking into account the monotony
assumptions, we easily show (for instance, by appealing to the first
order Taylor's formula with Lagrange form of the remainder) that the
Jacobian matrix $\,D\,A(p)\,$ of the transformation $\,A(p)\,$ is
positive semi-definite at each point $\,p\,\in\,\R^N\,$. So, for
each unit vector $\,\xi \in \,\R^N\,$,
\begin{equation}\label{quatroum}
DA(\,p)\, \xi \cdot\,\xi\leq\,tr\,DA(p), \quad \forall
\,\xi\in\,R^N\,,
\end{equation}
since the trace coincides with the sum of the eigenvalues.\par%
Let us denote the generical element of the Jacobian matrix
$\,D\,A(p)\,$ by $\,A_{i\,j}(p)\,.$ By setting $\,p=\,\na\,u\,$ one
has
\begin{equation}\label{amais}
div\,A(\,\na\,u(x)\,)=\,\sum_{i,\,j} \,A_{i\,j}(p) \,\pa_i\,p_j\,.
\end{equation}
Since $\,p=\,\al\,x\, r^{-\,1}\,$, it follows that
$\,\pa_i\,p_j=\,\al\ r^{-\,1} \,(\,\delta_{i\,j} -
\,r^{-\,2}\,x_i\,x_j\,)\,$. So, from \eqref{amais}, we get

\begin{equation}\label{quatrodois}\ba{ll}\vspace{1ex}\displaystyle%
div\,A(\,\na\,u(x)\,)=  \\
\dy  \al\ r^{-\,1} \, \big\{\,tr\,DA(\al\,x\,
r^{-\,1})-\,DA(\al\,x\,
r^{-\,1})\,(\,x\, r^{-\,1})\cdot\,(\,x\, r^{-\,1})\,\big\}\,,%
\ea%
\end{equation}
for each $\,x\neq\,0\,.$ From \eqref{quatrodois} and
\eqref{quatroum} it follows that $\,div\,A(\,\na\,u(x)\,)$ has the
sign of the constant $\,\al\,$, for each  $\,x\neq\,0\,.$\par%
Let $\,\phi\,$ be a non-negative, indefinitely differentiable
function in $\,\R^N\,.$ Fix $\,R>\,0\,$ such that
$$
\, supp \, \,\phi\,\subset\, I(0,\,R)\,.
$$
Next, fix a function $\gamma(x)\in \,D(\R^N)\,$ such that $\,0 \leq
\,\gamma(x)\leq\,1\,$, and $\,\gamma(x)=\,1\,$ for $\,|x|\leq\,1\,$.
To fix ideas, assume that $\,supp \,\, \gamma \,\subset I(0,\,2)\,.$
Further define, for each $s>\,0\,,$
$$
\gamma_s (x)=\,\gamma(s^{-\,1}\,x)\,, \quad \textrm{and} \quad
\phi_{s}(x)=\,\phi(x) \,(\,1-\,\gamma_{s}(x)\,).
$$
Note that, for all $s \in \,(0,\,R)\,,$
$$
supp\,\, \phi_{s} \subset \, I(0,\,R) -\,I(0,\,s)\,.
$$
Hence, by an integration by parts,
$$
\al\,\int_{I(0,\,R)} \,A(\na\,u(x)) \cdot \,\na\,\phi_s(x)
\,dx=\,-\,\al\, \int_{\complement I(0,\,s)}
\,div\,A(\na\,u(x))\,\phi_s(x) \,dx \leq\,0\,,
$$
where $\,u(x)=\,\al\,r\,$. Note that, on the left hand side, we may
replace $\,I(0,\,R)\,$ by $\,\R^N\,.$ We want to show that%
\begin{equation}\label{quatrotres}
\lim_{s \rightarrow \,0} \, \int_{I(0,\,R)} \,A(\na\,u(x)) \cdot
\,\na\,\phi_s(x) \,dx=\, \int_{I(0,\,R)} \,A(\na\,u(x)) \cdot
\,\na\,\phi(x) \,dx\,.
\end{equation}
This proves that
$$
\al \, \int_{\R^N} \,A(\,\na\,u)\cdot\,\na\,\phi \,dx \leq\,0\,,
$$
which is our thesis.\par%
Straightforward calculations shows that
\begin{equation}\label{eunem}
\na\,\phi_s(x)=\,(1-\,\gamma_s
(x)\,)\,\na\,\phi(x)-\,s^{-1}\,\phi(x)\,(\na\,\gamma)(s^{-1}\,x)\,.
\end{equation}
Since $\,(1-\,\gamma_s (x)\,)\,\na\,\phi(x)\,$ converges
point-wisely to $\,\na\,\phi(x)\,,$ $\,x\neq\,0\,,$ as $\,s
\rightarrow\,0\,,$ it readily follows, by Lebesgue dominated
convergence Theorem, that \eqref{quatrotres} holds by replacing, in
the left hand side, $\,\na\,\phi_s\,$ by $\,(1-\,\gamma_s
(x)\,)\,\na\,\phi(x)\,.$\par%

Let's see that in the left hand side of \eqref{quatrotres} the
contribution due to the second term in the right hand side of
\eqref{eunem} tends to zero. One has
$$
s^{-1}\,\int_{I(0,\,R)} \,|\,\phi(x)\,(\na\,\gamma)(s^{-1}\,x)\,|
\,dx \leq\,s^{-1}\,\int_{I(0,\,2)}
\,|\,\phi(s\,y)\,(\na\,\gamma)(y)\,| s^N \,dx\,
$$
$$
\phantom{aaaaaaaaaaaaaaaaaaaaaaaaaa}\leq 2^N\,V_N\,
s^{N-\,1}\,\|\,\phi\,\|_{L^{\infty}(\R^N)}\,\|\,\na\,\gamma\,\|_{L^{\infty}(\R^N)}
\,,
$$
where $\,V_N\,$ denotes the volume of the unit sphere. Since $
\,A(\na\,u(x))\,$ is uniformly bounded in $\,{I(0,\,R)}\,$, the
thesis follows.

\vspace{0.2cm}

Assume now that $\,A(p)$ is merely continuous. Let
$\,j_{\ep}(\eta)\,$ be, for each $\ep>\,0\,,$ a real, nonnegative
function, indefinitely differentiable, with compact support
contained in the sphere $\,I(0,\,\ep)\,$, and integral equal to
$\,1\,$. Set
$$
A_{\ep}(p)=\,\int \,A(\eta)\,j_{\ep}(p-\,\eta\,)\,d\eta\,.
$$
These functions are indefinitely differentiable. Furthermore,
\begin{equation}\label{quatroquatro}
A_{\ep}(p)-\,A_{\ep}(q)=\,\int\,\big[\,A(p-\,\xi)-\,A(q-\,\xi)\,\big]\,j_{\ep}(\xi)\,d\xi\,.
\end{equation}
In particular, this last inequality implies that $\,A_{\ep}(p)\,$
satisfies the monotony hypothesis \eqref{zerotres} (note that
assumptions \eqref{zerodois}, \eqref{zeroquatro}, and
\eqref{zerocinco} were not used here). From the first part of the
proof it follows that
\begin{equation}\label{quatrocinco}
\al \, \int \,A_{\ep}(\,\na\,u)\cdot\,\na\,\phi \,dx \leq\,0\,,
\end{equation}
for all nonnegative $\,\phi \in \,D(\Om)\,.$ Further, since%
$$
A_{\ep}(p)-\,A(p)=\,\int\,\big[\,A(\eta)-\,A(p)\,\big]\,j_{\ep}(p-\,\eta)\,d\eta\,,
$$
and since $\,A(p)\,$ is uniformly continuous on compact sets, it follows that
$\,A_{\ep}(p) \rightarrow \,A(p)\,$ uniformly on compact sets. So, letting $\,\ep \rightarrow\,0\,$ in equation
\eqref{quatrocinco}, one gets the thesis.
\end{proof}

\vspace{0.2cm}

The next result concerns the local character of the notion of regular point.
\begin{theorem}\label{teoquatrodois}
let $\Om\,$ and $\,\Lambda\,$ be two open bounded sets, and let $\,y\,\in\, \pa\,\Om \cap\,\pa\,\Lambda\,.$
Assume, moreover, that there exists a sphere $\,I(y,\,r)\,$ such that
\begin{equation}\label{quatroseis}
I(y,\,r)\cap\,\Om=\,I(y,\,r)\cap\,\Lambda\,.
\end{equation}
Then $\,y\,$ is regular with respect to $\,\Om\,$ if and only if it
is regular with respect to $\,\Lambda\,.$
\end{theorem}
\begin{proof}
Due to Theorem \ref{teo-A}, it is sufficient to show that there is a
system of barriers with respect to $\,\Lambda\,$ if and only if
there is a system of barriers with respect to $\,\Om\,.$\par%
Let $y$ be regular with respect to $\,\Lambda\,.$ Assume, for the
time being, that $\,\Lambda \subset \Om\,$. Given $\,\rho\,$ and
$\,m\,$, $\,0<\,\rho<\,r\,$ and $\,0<\,m\,$, let $\,V(x)\,$ be the
variational solution in $\,\Om\,$ to the problem \eqref{zeroseis}
with boundary data given by
$$
h(x)=\,m\,|\,x-\,y\,|^2\,\rho^{-\,2}\,.
$$
By the construction, $\,V(x)\,$ satisfies the condition (j) in
definition \ref{defzeroum}. Let us show that it also satisfies
condition (jj). Let be
$$
M\geq\,max\{\,1,\,m^{-\,1}\,Sup_{\,\Om}\,|\,V(x)\,|\,\}\,,
$$
and let $\,V'(x)\,$ be the solution in $\,\Lambda\,$ with boundary
data $\,h'(x)=\,M\,m\,|\,x-\,y\,|^2\,\rho^{-\,2}\,$. Clearly,
$\,V(x)\,$ is a solution in $\,\Lambda\,$. Furthermore, from the
definition of $\,M\,$, it follows that $\,V'\geq\,V\,$ on
$\,\pa\,\Om\,.$ from this last inequality, together with lemma
\ref{lemumtres}, we show that $\,V'(x)\geq\,V(x)\,$ almost
everywhere in  $\,\Lambda\,$. From this last assertion, together
with the regularity of $\,y\,$ with respect to $\,\Lambda\,$, it
follows that
$$
0\leq\,\lim_{ x\in\,\Om ,\, x\rightarrow\,y}\,V(x) \leq\,\lim_{
x\in\,\Lambda ,\, x\rightarrow\,y}\,V'(x)=\,0\,.
$$
This proves the assumption (jj). By appealing to the theorem
\ref{teo-A} we conclude that $\,y\,$ is regular with respect to
$\,\Om\,$. The existence of the function $\,U(x)\,$ referred in
definition \ref{defzeroum} may be shown by a similar argument, or by
appealing to remark \ref{rem-1.1}.\par%
Reciprocally, assume that $\,y\,$ is regular with respect to
$\,\Om\,$. Given $\,\rho\,$ and $\,m\,$, $\,0<\,\rho<\,r\,$ and
$\,0<\,m\,$, we construct below the corresponding barrier $\,V(x)\,$
in $\,\Lambda\,$, according to definition \ref{defzeroum}. Let
$\,V(x)\,$ be the solution in $\,\Om\,$ with boundary data
$\,h(x)=\,m\,|\,x-\,y\,|\,\rho^{-\,1}\,$ on  $\,\pa\,\Om\,.$
$\,V(x)\,$ is a solution in  $\,\Lambda\,$, and satisfies the
condition (ii) since $y$ is regular with respect to $\,\Om\,$ and
$\,h(x)=\,0\,$. Further, since $\,V=\,h\,$ on $\,\pa\,\Om\,$ and
$\,h(x)\,$ is a sub-solution in $\,\Om\,$ (lemma \ref{lemquatroum}),
it must be  $\,V(x) \geq\,h(x)\,$ almost everywhere in $\,\Om\,$.
In particular, $\,V \geq\,h\,$, so  $\,V \geq\,m\,$, on $\,\pa\,\Lambda \cap\,\complement I(y,\,\rho)\,,$ as desired.\par%
Finally, if $\Lambda\,$ is not contained in $\,\Om\,$, consider the open set $\,D=\,I(y,\,r)\cap\,\Lambda=\,I(y,\,r)\cap\,\Om\,,$
and take into account that $\,D\subset\,\Lambda\,$ and  $\,D\subset\,\Om\,.$
\end{proof}

\vspace{0.2cm}

We end this section by proving the theorem \ref{teo-B}.%

\vspace{0.2cm}

%\begin{proof}
\emph{Necessary condition}: Let $\,y\,$ be regular. By theorem
\ref{teoquatrodois} it follows that $\,y\,$ is regular with respect
to $\,\Si -\,E_{\rho}\,.$ Since the capacitary potential
$\,u_{\rho,\,m}\,$ is the solution in $\,\Si -\,E_{\rho}\,$ (lemma
\ref{lemumseis}) with data $\,m\,$ on $\,\pa\,E_{\rho}\,$ and
$\,0\,$ on $\,\pa\,\Si\,$ (lemma \ref{lemumcinco}), the first
equation \eqref{zerodezasseis} follows. A similar argument applies
to $\,u_{\rho,\,-\,m}\,.$%

\vspace{0.2cm}

\emph{Sufficient condition:} We assume that the hypothesis
\ref{zerodezasseis} holds, and we prove the existence of a system of
barriers at $\,y\,$. Given $\,\rho>\,0\,$ and $\,m>\,0\,,$ we
construct the function $\,V(x)\,$ referred in the definition
\ref{defzeroum}. Let $\,R_0\,$ be such that $\,\Si \subset
\,I(y,\,R_0)\,$, and define $\,k>\,0\,$ by
\begin{equation}\label{quatrosete}
\frac{(k+\,2\,m)\,\rho}{2\,R_0}=\,m\,.
\end{equation}
For convenience, we denote by $\,u\,$ the capacitary potential
$\,u=\,u_{\frac{\rho}{2},\,-(m+\,k)}\,.$ Furthermore, we define in
$\,\Si\,$ the function $\,V=\,u+\,(m+\,k)\,$. $\,V\,$ is a solution
in $\,\Si-\,E_{\frac{\rho}{2}}$ (lemma \ref{lemumseis}) and, in
particular, it is a solution in $\,\Om\,.$ Since $\,\lim_{x
\rightarrow\,y}\,u(x)=\,-\,(m+\,k)\,,$ $\,V\,$ satisfies the
condition (jj) in definition \ref{defzeroum}. Obviously
$\,V(x)\geq\,0\,$ a.e. in  $\,\Si\,$, as follows from lemma
\ref{lemumcinco}.

\vspace{0.2cm}

Next we prove the condition (j). Consider in $\,I(y,\,R_0)\,$ the
function $\,f(x)=\,(k+\,2\,m)\,
R_0^{-\,1}\,|\,x-\,y\,|-\,(k+\,2\,m)\,.$ This function is a
sub-solution in $\,I(y,\,R_0)\,$ (lemma \ref{lemquatroum}) and, in
particular, is a sub-solution in
$\,I(y,\,R_0)-\,I(y,\,\frac{\rho}{2})\,$. Since $\,\Si
\subset\,I(y,\,R_0)\,,$ it follows that $\,f\leq\,0\,,$ on
$\,\pa\,\Si\,.$ Further, from \eqref{quatrosete}, it follows that
$\,f=\,-\,(k+\,m)\,$ on $\,\pa\,I(y,\,\frac{\rho}{2})\,.$ So
\begin{equation}\label{quatrooito}
f\leq\,u \quad \textrm{on} \quad \pa\,I(y,\,\frac{\rho}{2})\,.
\end{equation}
By appealing to \eqref{quatrooito}, to the inequality $\,f\leq\,0\,$
on $\,\pa\,\Si\,$, and to the lemma \ref{lemumtres} applied in
$\,\Si-\,I(y,\,\frac{\rho}{2})\,,$ it follows that $\,f(x)
\leq\,u(x)\,$ almost everywhere in this last set. So,
\begin{equation}\label{quatronove}
V(x) \geq\,f(x)+\,(m+\,k) \geq\,m, \quad \textrm{ a.e. on} \quad \Si-\,I(y,\,\rho)\,.
\end{equation}
In particular, \eqref{quatronove} implies that $\,V\geq\,m\,$ on $
\pa\,\Om-\,I(y,\,\rho)\,,$ hence the condition (j) holds.%
\part{}
\section{Main results}\label{intpart2}
We start by remarking that the proofs presented in Part II strongly
rely on ideas and techniques used in reference \cite{b12}, to
which the reader is referred.\par%
The aim of the second part of this work is to state sufficient
conditions for regularity of a given boundary point $\,y\,$. This
task is done by appealing to the theorem \ref{teo-B}. The sufficient
conditions obtained here consist in assumptions on the sets%
\be\label{erocompl}%
E_{\rho}=\,(\complement \,\Om)(y,\,\rho)\,,%
\ee%
the complementary sets of $\,\Om\,$ with respect to the closed balls
$\,\overline{I(y,\,\rho)}\,.$ They always concern sufficiently small
values of the radius $\,\ro\,.$\par%
The cornerstone result of part II, the theorem \ref{teo-nopub}, has
an "abstract" feature due to the assumption \eqref{e62b}. However we
show that this assumption holds if simple geometrical conditions are
fulfilled. This leads to the statements in theorems
\ref{teo-unpub} and \ref{teo-bvc} below.\par%
\begin{definition}\label{assumpsigma}
Let $\,y \in\,\pa\,\Om\,$ be a boundary point. Given $\,\ro>\,0\,$,
we denote by $\,\hs(\ro)\,$ a positive real such that
the estimate%
\be\label{e62}%
|\,v(x)\,|\leq \,\hs(\ro)^{\,-\,1} \,\int_{I(y,\,\ro\,)} \,
\frac{|\,\na\,v(z)\,|}{|\,x-\,z\,|^{\,N-\,1}}\, dz%
\ee%
holds almost everywhere in $\,I(y,\,\ro\,)\,$, for all
$\,v\in\,H^{\,1,\,t}(I(y,\,\ro\,)\,)\,$ vanishing identically on
$\,E_{\rho}\,.$\par%
If such a positive value does not exist, we set
$\,\hs(\ro)=\,0\,.$ %
\end{definition}

We assume that there is a strictly positive function
$\,\sigma(\ro)\,$, and a constant $\,C\,$ such that, for each
positive $\,\ro\,$ in a arbitrarily small neighborhood of zero.

\vspace{0.2cm}

The next theorem, and the related theorems \ref{teo-unpub} and
\ref{teo-bvc} below, are the main results in part II.
\begin{theorem}\label{teo-nopub}
There is a positive constant $\,\Lambda\,$, which depends only on
$\,t,\,N,\,a,\,$ and $p_0$, such that if%
\be\label{e62b}%
\big[\,\hs(\rho)\,\big]^{\frac{t}{t-\,1}}\geq\,\Lambda\,(\log\,\log\,\rho^{-\,1}\,)^{-1}\,,%
\ee%
for small, positive, values of $\,\rho\,,$ then the point
$\,y\in\,\pa\,\Om\,$ is regular with respect to the operator $\,\cL\,$ .%
\end{theorem}
The next two theorems are corollaries of the theorem
\ref{teo-nopub}.
\begin{theorem}\label{teo-unpub}
There is a positive constant $\,\Lambda\,$, which depends only on
$\,t,\,N,\,a,\,$ and $p_0$, such that  if
$$
\Big(\,\frac{|\,E_\rho\,|}{|\,I(y,\,\rho)\,|}\,\Big)^{\frac{t}{t-\,1}}\geq\,\Lambda\,(\log\,\log\,\rho^{-\,1}\,)^{-1}\,,
$$
for small, positive, values of $\,\rho\,,$ then the point
$\,y\in\,\pa\,\Om\,$ is regular with respect to the operator $\,\cL\,$ .%
\end{theorem}
Note that this condition is stronger then the usual cone condition
since the right hand side goes to zero with $\,\ro\,.$\par%
The next statement is theorem 5.5 in reference \cite{b1} (see also
\cite{ricmat}, page 5).
\begin{theorem}\label{teo-bvc}
A point $\,y\in\,\pa\,\Om\,$ is regular with respect to the operator
$\,\cL\,$ if $\,y\,$ satisfies a $\,(N-\,1)-$dimensional external
cone property.  The $\,(N-\,1)-$dimensional external cone property
may be replaced by a generalized $\,(N-\,1)-$dimensional external
cone property.
\end{theorem}
In this section, by assuming that theorem \ref{teo-nopub} holds, we
prove theorems \ref{teo-unpub} and \ref{teo-bvc}. This is done by
proving that \eqref{e62b} follows from the geometrical assumptions
required both in theorem \ref{teo-unpub} and in theorem
\ref{teo-bvc}. So, as soon as this purpose is fulfilled, our task
will be to present the proof of theorem \ref{teo-nopub}. This proof
is postponed to the next two sections.\par%
We start  by theorem \ref{teo-unpub}. This theorem follows
immediately from theorem \ref{teo-nopub}, by appealing to the
following result.
\begin{lemma}\label{lestamfourier}
Let $\,\sigma(\rho)\,$ be defined by \eqref{densmedida} and assume
that $\,|\,E_{\rho}\,|>\,0\,.$ Then
$$
\hs(\ro)^{-\,1} \leq\,C\, \sigma(\rho)^{-\,1}\,,%
$$
where $\,C=\,\frac{\,2^N}{N\,V_N}\,.$
\end{lemma}
Let us prove this lemma. The proof strictly follows the proof of
theorem 6.2, shown in reference \cite{b12}. We denote by $\,S\,$ the
surface of the $\,N$ dimensional unit sphere. Further, if
$\Th\subset\,S$, we denote by $\,|\,\sph\Th\,|\,$ the $\,(N-\,1)-$\~
dimensional spherical measure of $\,\Th\,$,
$$
|\,\sph\Th\,|=\,\int_{\Th} \,dS\,.
$$
\begin{lemma}\label{misas}%
Set $\,I=\,I(y,\,\ro)\,,$ and let $\,E=\,E_\ro\,$ be given by
\eqref{erocompl}. Furthermore, let a point $x\in I\,, $ $x\notin
E\,$, be given, and denote by $S$ the surface of the unit sphere
centered in $x$. Finally, consider the set
$$
\Th=\,\Th(x)=\,\{\xi \in S:\,\exists \,t=\,t(\xi)\in\,\R
\,,\,x+\,t\,\xi \in\,E\,\}\,.
$$
Then the estimate
$$
|v(x)|\leq\,\frac{1}{|\sph\Th(x)|}\,\,\int_{I} \,
\frac{|\,\na\,v(z)\,|}{|\,x-\,z\,|^{\,N-\,1}}\, dz%
$$
holds for any function $v\in\,C^1(I)$ vanishing on $E$.
\end{lemma}
\begin{proof}
Let $\xi \in \,\Theta\,$ and $t(\xi)\in\,\R\,$ be such that
$\,x+\,t(\xi)\,\xi \in\,E\,$. Since
$$
|\,v(\,x+\,t(\xi)\,\xi\,)-\,v(x)\,|\leq\,\int_{0}^{t(\xi)}
\,|\,\na\,v(x+\,r\,\xi)\,|\,dr\,,
$$
and $\,|x-\,z|^{N-\,1}\,dS\,dr=\,dz\,,$ it follows that
$$
|\,\sph\,\Theta\,|\,|v(x)|=\,\int_{\Theta}\,|v(x)|\, dS
\leq\,\int_{\Th}\,\int_{0}^{t(\xi)}
\,|\,\na\,v(x+\,r\,\xi)\,|\,dr\,\,dS\leq\, \int_{I} \,
\frac{|\,\na\,v(z)\,|}{|\,x-\,z\,|^{\,N-\,1}}\, dz\,.%
$$
\end{proof}
\begin{corollary}\label{misacorol}%
Let $v\in\,C^1(I)$ vanish on $E$. Assume that
\be\label{liminf}%
|\,\sph\,\Th\,|\equiv \inf_{x}\,|\,\sph\,\Th(x)\,| \,>\,0\,.
\ee%
Then \be\label{liminfb}%
|v(x)|\leq\,\frac{1}{|\,\sph\,\Th\,|}\,\,\int_{I} \,
\frac{|\,\na\,v(z)\,|}{|\,x-\,z\,|^{\,N-\,1}}\, dz\,,%
\ee%
for all $x\in\,I\,.$ Furthermore, if $\,v \in H^{1,\,t}(I)\,$
vanishes on $\,E\,$ in the $\,H^{1,\,t}(I)\,$ sense, then
\eqref{liminfb} holds a.e. in $\,I\,.$
\end{corollary}
Note that the estimate \eqref{liminfb} is obvious if $x\in \,E\,.$
The last assertion in the corollary follows from well known results
on the continuity of the linear map defined by convolution with the
kernel $\,|\,z\,|^{-\,(N-\,1)}\,$. Actually, this map is continuous
from $\,\L^r\,$ to $\,L^{r^*}\,$ where $1/r^{r^*}=1/r-\,1/n\,.$ See,
for instance, \cite{stein}, Chap.V.\par%
Lemma \ref{lestamfourier} follows by appealing to the "volumetric"
estimate
$$(1/N) \,|\sph\,\Th(x)| \, (\textrm{diam}\,I)^N \geq\,|E|\,.$$
%
%\begin{corollary}\label{misacorolb} Assume that \eqref{liminf} holds. Then
%$$
%|v(x)|\leq\,C\,\frac{|\,\Si\,|}{|\,E\,|}\,\int_{\Si} \,
%\frac{|\,\na\,v(z)\,|}{|\,x-\,z\,|^{\,N-\,1}}\, dz\,,%
%$$
%for all $x\in\Si\,,$ where
%$$
%C=\,\frac{\,(\textrm{diam}\Si)^N}{N\,|\Si|}\,.
%$$
%\end{corollary}
%

\vspace{0.2cm}

Next, we prove theorem \ref{teo-bvc}. Some details are left to the
reader. By "cone" (in any dimension) we mean a right circular cone,
truncated by a sphere with center the vertex of the cone. For
instance, the $\,(N-\,1)-$dimensional "truncated cones" with vertex
$\,y=\,0\,$ have the form
\begin{equation}\label{cumum}
C_{\rho,\,\om}=\,\{\,x\in\,\R^N\,:\,x_1\geq\,0\,,\,x_N=\,0\,,
\,|\,x\,| \leq\,\rho\,,\, |\,x\,|^2 \leq\, (1+\,\om)\,{x_1}^2\,\}\,,
\end{equation}
where $\,\rho\,$ and $\,\om \,$ are positive constants. Note that,
by setting $\,x=\,(x_1,\,x',\,x_N\,)\,,$ the above
condition means that $\,|\,x'\,|^2 \leq\,\,\om\,\,x^2_1\,.$\par%
\begin{definition}\label{s-cone}
We say that a point $\,y\in\,\pa\,\Om\,$ satisfies an
$\,(N-1)-$dimensional external cone property if there exists an
$\,(N-1)-$dimensional cone $\,C\,$ with vertex in $\,y\,$ and
contained in $\,\complement \,\Om\,.$ Similarly, we define
generalized $\,(N-1)-$dimensional cone property at the point
$\,y\,$, by replacing the cone $\,C\,$ by a Lipschitz image of
itself.
\end{definition}
%Note that given a cone $\,C\,$ there is a "small" cone $\,C_0\,$
%such that to each point $\,z\in\,C\,$ it corresponds a cone
%$\,C_0(z)\subset \,C\,$, with vertex in $\,z\,,$ and congruent to $\,C_0\,$.%

\vspace{0.2cm}

The proof of theorem \ref{teo-bvc} follows immediately from theorem
\ref{teo-nopub} and corollary \ref{misacorol}, by a small
modification of the argument used to prove the theorem
\ref{teo-unpub}. As above, we appeal to the corollary
\ref{misacorol}. Roughly speaking, as for the theorem
\ref{teo-unpub}, we would like to show that there is a positive
lower bound $\,|\,\sph\,\Th\,|\,$ for the values of the solid angles
$\,|\,\sph\,\Th(x)\,|\,$ from which the set $\,E_\ro\,$ can be
"watched" from points $\,x\in I(y,\,\ro)\,.$ Clearly, this is false
in general, since (for instance) $\,x\,$ and $\,E_\ro \,$ may belong
to a $\,(N-\,1)-$dimensional hyperplane. However the same argument
applies here. Let's prove that equation \eqref{e62} holds for a
positive $\,\sigma(\ro)\,,$ independent of $\,\ro\,.$ To show this
claim, note that geometry and estimates for a generical value
$\,\ro\,$ can immediately be brought back to the case $\,\ro=\,1\,$,
by a suitable homothety. Next, note that the estimates in play are
invariant under Lipschitz maps, up to multiplication by positive
constants. So, we may fold up the original  $\,(N-1)-$ dimensional
cone into an "non flat" $\,(N-1)-$ dimensional "twisted cone", which
contains $\,N\,$ distinct pieces of surface, each one orthogonal to
a single $\,x_i\,$ direction, $\,i=\,1,...,\,N\,.$ Now, from each
point $\,x \in I(y,\,1)\,,$ one "watches", at least, one of the
above pieces of surface, from a positive solid angle
$\,|\,\sph\,\Th(x)\,|\,.$ Moreover, the lower bound
$\,|\,\sph\,\Th\,|\,$ of the values of solid angles is positive.
This proves theorem \ref{teo-bvc}.\par%
Note that it would be sufficient to prove that the lower bounds
behaves like $\,\sigma(\rho)\,$ in equation \eqref{densasr}, as
$\,\ro\,$ goes to zero.
\section{A recursive estimate for the local oscilation}\label{regcrit}
In the sequel, to avoid unessential devices, we assume in equations
\eqref{zeroquatro} and \eqref{zerocinco} that $\,p_0=\,0\,$. One
easily extends the proof to the general situation by appealing to
\eqref{doistres}. This leads to the appearance of "lower order"
terms, easy to control.

\vspace{0.2cm}

We prove the theorem \ref{teo-unpub} by showing that
\eqref{zerodezasseis} holds. More precisely, we fix a couple of
positive constants $\,\rho_0\,$ and $\,m\,,$ and prove that
$$
\lim_{x \rightarrow \,y} u_{m,\,\rho_0}(x)=\,m\,.
$$
The proof of the second equation \eqref{zerodezasseis} is absolutely
identical. Alternatively, we may appeal to the remark \ref{rem-1.1},
to refer the proof to that of the first equation.\par
\vspace{0.3cm}

In the sequel the "large" ball $\,\Si\,$, the point $\,y\in \pa
\,\Om\,$, and the positive constants $m$ and $\,\rho_0\,$ are
assumed to be fixed, once and for all. The capacitary potential
$u_{m,\,\rho_0}(x)$ of $\,E_{\rho_0} \,$ will be simply denoted by
$u(x)$. Furthermore, without loss of generality, we place the
origin at $\,y\,$, so%
$$
y=\,0\,.
$$
We set $\,I(r)=\,I(0,\,r)\,.$
The following result is well known.
\begin{lemma}\label{L51}
One has
\begin{equation}\label{52}
\|\,v\,\|_{t^*,\,r} \leq\,c\,\|\,\na\,v\,\|_{t,\,r}\,,\quad \forall
\, \,v \in H^{1,\,t}_0(r)\,,
\end{equation}
where $ \frac{1}{t^*}=\,\frac{1}{t}-\,\frac{1}{N}\,. $
\end{lemma}
We define sets
\be\label{asbas}%
\,B(k,\,r)= \,\{x \in\,\Om(y,\,r)\,: u(x) \leq\,k\,\}\,,%
\ee%
and introduce the cut-off function
\be\label{fixis}%
\phi(x)=\, \left\{\begin{array}{ll}\dy \,1
&\dy \mbox{ if }\ |x|\leq\,\rho\,,\\
\dy \frac{R-\,|x|}{R-\,\rho} & \dy \mbox{ if }\ \rho
\leq\,|\,x\,|\leq\,R\,,\\%
\,0 &\dy \mbox{ if }\ R \leq\,|x|\,.
\end{array}\right.\ee
In the sequel, $\,0<\,\rho <\,R <\,\ro_0\,.$ For brevity, we set
$$
%A(k)=\,A(k,\,R)\,,\quad%
B(k)=\,B(k,\,R)\,.
$$
The following kind of estimates is well known.
\begin{lemma}\label{L52}
Assume that $\,0<\,\rho\, <R <\,\overline{R}\,,$ and
$\,0<\,h<\,k\,.$ Let be $\,v \in \, H^{1,\,t}(\overline{R})\,.$
Then, the following estimates hold.
\be\label{exle64}%
\left\{\begin{array}{ll}\dy
\int_{B(h,\,\rho)} (h-\,u)^t \,dx\leq\\%
c\,\Big(\,(R-\,\ro)^{-\,t} \int_{B(k)} \,(k-\,u)^t \,dx\,+
\int_{B(k)} \, |\,\na\,u\,|^t \,\phi^{\,t}
\,dx\,\Big)\,|B(k)|^\frac{t}{N}\,,\\
\\
\dy |B(h,\,\rho)|\,(k-\,h)^t \leq \, \int_{B(h,\,\ro)} \,(k-\,u)^t
\,dx\, \leq \, \int_{B(k)} \,(k-\,u)^t \,dx\,.
\end{array}\right.\ee
\end{lemma}
For the proof of the first estimate see, for instance, the proof of
the first inequality (6.12) in reference \cite{c2}. The second
estimate \eqref{exle64} is obvious.
\begin{theorem}\label{T51}
Let $\,\phi\,$ be given by \eqref{fixis}. Then, for each real $k$,%
\be\label{eq54}%
\int_{B(k,\,R)}\, |\,\na\,u\,|^t \,\phi^{\,t}\,dx \leq\,
c\,(R-\,\ro)^{-\,t} \int_{B(k,\,R)} \,|\,u-\,k\,|^{\,t} \,dx\,.%
\ee
\end{theorem}
\begin{proof}
By the definition of $u_{m,\,\rho_0}(x)\,$ one has
\begin{equation}\label{Eq55}
\int_{\Si}  \,\big(\,A(\na\,u\,),\,\na\,(\,v-\,u\,)\,\big) \,dx
\geq\,0\,, \quad \forall \quad v \in\,\K_m(\Sigma)\,.
\end{equation}
where (recall \eqref{zeroonze})
$$
\K_m(\Sigma)=\,\big\{\,v \in\,H^{1,\,t}_0(\Si)\,:\,v \geq\,m \quad
\mbox{ on }\quad E_{\ro_0}\,\big\}\,.
$$
By setting in equation \eqref{Eq55} $\,v=\,u-\,\phi^{\,t}
\,\min(u-\,k,\,0)\,$ it follows that
\begin{equation}\label{Eq56}
\int_{B(k)}  \,\big(\,A(\na\,u\,),\,\na\,u\,\,\big) \,\phi^t \,dx
\leq\,-t\,\int_{B(k)}
\,\big(\,A(\na\,u\,),\,\na\,\phi\,\big)\,(\,u-\,k)\,\phi^{\,t-\,1}
\,dx.
\end{equation}
From \eqref{Eq56}, by appealing to H\H older's inequality and to
properties enjoyed by $\,\phi\,$ and $\,A(p)\,,$ we show that
\be\label{eq57}\begin{array}{ll}\dy%
a\, \int_{B(k)} \, |\,\na\,u\,|^t \,\phi^{\,t} \,dx \leq \\
\\
t\,a^{\,t-1}\,\Big(\,\int_{B(k)} \,|\,\na\,u\,|^t \,\phi^{\,t}
\,dx\,\Big)^{\frac{t-\,1}{\,t}}\, \Big(\, \int_{B(k)} \,
|\,u-\,k\,|^t \,|\,\na\,\phi\,|^{\,t}
\,dx\,\Big)^{\frac{\,1}{\,t}}\,.%
\ea\ee%
Equation \eqref{eq57} leads to
$$
\int_{B(k)} \, |\,\na\,u\,|^t \,\phi^{\,t} \,dx \leq\,c\,
\int_{B(k)} \, |\,u-\,k\,|^t \,|\,\na\,\phi\,|^{\,t} \,dx\,.
$$
Since $\,|\,\na\,\phi\,| \leq\,(R-\,\ro\,)^{\,-1}\,,$ the thesis
follows.
\end{proof}
The next result follows by appealing to theorem \ref{T51} and lemma
\ref{L52}.
\begin{lemma}\label{L53}
Assume that $\,0<\,\rho\, <R\,,$ and $\,0<\,h<\,k\,.$ The following
estimates hold.
\be\label{exle64}%
\left\{\begin{array}{ll}\dy \int_{B(h,\,\rho)} (h-\,u)^t \,dx\leq \,
c_1\,|B(k)|^\frac{t}{N}\,(R-\,\ro)^{-\,t} \int_{B(k)} \,(k-\,u)^t \,dx\,,\\
\\
\dy |B(h,\,\rho)|\,(k-\,h)^t \leq \, \int_{B(k)} \,(k-\,u)^t \,dx\,.
\end{array}\right.\ee
\end{lemma}
For brevity we set
\be\label{exle64b}%
\left\{\begin{array}{ll}\dy%
u(h,\,\rho)=\,\int_{B(h,\,\rho)} (h-\,u)^t \,dx\,,\\
\\
\dy b(h,\,\rho)=\,|B(h,\,\rho)|\,.
\end{array}\right.\ee
So, equation \eqref{L53} takes the form
\be\label{exle64b}%
\left\{\begin{array}{ll}\dy%
u(h,\,\rho)\leq \,
c_1\,b(k,\,R)^\frac{t}{N}\,(R-\,\ro)^{-\,t}\,u(k,\,R)\,,\\
\\
\dy b(h,\,\rho)\,(k-\,h)^t \leq \,u(k,\,R)\,.
\end{array}\right.\ee
Next, we define
\be\label{psistamp}%
\psi (h,\,\rho)=\,u(h,\,\rho)^{\,\theta \frac{N}{t}} \, b(h,\,\rho)\,,%
\ee%
where
$$
\theta=\,\frac12 +\,\sqrt{\frac{1}{4}+\,\frac{t}{N}}>\,1\,.
$$
Straightforward calculations show that
\be\label{psisminus}%
\psi(h,\,\rho)\leq\,c_1^{\,\frac{N}{t}\,\theta\,}\,
\,\frac{1}{(\,R-\,\ro\,)^{\,N\,\theta}}\,
\frac{1}{(\,k-\,h\,)^{\,t}}\,\psi(k,\,R)^\theta \,.%
\ee%
Note that
$\,\frac{t}{N}+\,\theta=\,\theta^{\,2}\,.$ We point out that the above choice
of $\,\theta\,$ is the only choice possible to get an estimate of the form \eqref{psisminus}.\par%

\vspace{0.2cm}

\begin{lemma}\label{L55bis}
Let be $\,0<\,r_0 \leq\,\frac{\,\ro_0}{2}\,,$ $\,k_0\in\,\R\,,$ and
$\,d>\,0\,.$ Define, in correspondence to each no-negative integer
$m$, the following quantities:
\be\label{e514bis}%
\left\{\begin{array}{ll}%
r_m=\,\frac{r_0}{2}+\,\frac{r_0}{2^{m+1}}\,,\\
\\
k_m=\,k_0-\,d +\,\frac{d}{2^{\,m}}\,,%
\ea\right.%
\ee
\be\label{e515bis}%
\left\{\begin{array}{ll}%
a_m=\,|B(k_m,\,r_m)|\,,\\
\\
u_m=\,\int_{B(k_m,\,r_m)} \,(k_m-\,u)^t \,dx\,,%
\ea\right.%
\ee
and
\be\label{e516biss}%
\psi_m=\,u_m^{\,\theta \frac{N}{t}} \, b_m\,.%
\ee%
Then
\be\label{demeias}%
\big|B(k_0-\,d,\,\frac{r_0}{2})\big|=\,0%
\ee%
if
\be\label{ddt}%
d \geq c_1^{\,\frac{N \,\theta}{\,t^2} }\, \frac{
2^{\frac{\beta\,\theta}{t}}}{(2\,r_0)^{\frac{\,N\,\theta}{t}}}\,
\,\psi_0^{\frac{\,\theta-1}{t}}\equiv \,
C\,\frac{\psi_0^{\frac{\,\theta-1}{t}}}{r_0^{\frac{\,N\,\theta}{t}}}\,.%
\ee%
\end{lemma}
\begin{proof}
Note that $\,a_m\,,\,u_m\,,$ and  $\,\psi_m\,,$ are non-increasing
sequences. By setting in equation \eqref{psisminus}
$\,(k,\,R)=\,(k_m,\,r_m)\,,$ and $
\,(h,\,\ro)=\,(k_{m+\,1},\,r_{m+\,1})\,,$
one shows that%
\be\label{bicas}%
\psi_{m+1}\leq\,c_1^{\,\frac{N}{t}
\theta\,}\,\frac{1}{d^t}\,\frac{1}{(2\,r_0)^{\,N\,\theta}}\,
\,2^{(\,m+\,1)\,(\,t+\,N\,\theta)}\,\psi_m^{\theta}\,.
\ee%
We want to prove, by induction, that
\be\label{inducao}%
\psi_m\leq\,\frac{\psi_0}{2^{\,\beta\,m}}\,,\quad \forall \,
m\geq\,0\,,
\ee%
where
$$
\beta=\,\frac{\,t+\,N\,\theta}{\theta-\,1}\,.
$$
For $m=0\,,$ \eqref{inducao} is obvious. Assume it for some
$\,m\geq\,0\,.$ By appealing to \eqref{bicas} and \eqref{inducao}
straightforward calculations show that

\be\label{bicasb}%
\psi_{m+1}\leq\,c_1^{\,\frac{N}{t}
\theta\,}\,\frac{1}{d^t}\,\frac{2^{\beta\,\theta}}{(2\,r_0)^{\,N\,\theta}}\,
\,\psi_0^{\theta-\,1}\,\frac{\psi_0}{2^{\,\beta\,(m+\,1)}}\,.
\ee%
This proves \eqref{inducao}, under the assumption \eqref{ddt}. In
particular, $\,\psi_m \rightarrow \,0\,,$ as $\,m \rightarrow
\,\infty\,.$ Since
$$
\big|\,B(k_0 -\,d,\,\frac{r_0}{2})\,\big| \,\Big\{\,\int_{B(k_0
-\,d,\,\frac{r_0}{2})} \,\big(\,(\,k_0 -\,d\,)-\,u\,\big)^t
\,dx\,\Big\}^{\,\theta \frac{N}{t}} \leq\,\psi_m\,,
$$
the thesis of the theorem follows.%
\end{proof}

\begin{corollary}\label{corinf}%
There is a constant $C$, independent of $r_0$ and $k_0$, such that
\be\label{minimus}%
\textrm{Inf}_{I(\frac{r_0}{2})}  \,u \geq\,k_0 -\,
C\,\Big\{\,\frac{1}{r_0^N}\,\int_{B(k_0,\,r_0)} \,
\big(\,k_0-\,u\,\big)^t \,dx\,\Big\}^\frac{1}{t}\,
\Big\{\,\frac{1}{r_0^N}\,|\,B(k_0,\,r_0)\,|\Big\}^\frac{\theta-1}{t}\,.
\ee%
In particular,
\be\label{minimus2}%
\textrm{Inf}_{I(\frac{r_0}{2})}  \,u \geq\,k_0 -\,
C\,\Big\{\,\frac{1}{r_0^N}\,\int_{B(k_0,\,r_0)} \,
\big(\,k_0-\,u\,\big)^t \,dx\,\Big\}^\frac{1}{t}\,.
\ee%
\end{corollary}
The proof of the first estimate follows immediately from
\eqref{demeias}, by taking into account that the $C$ term in the
right hand sice of \eqref{minimus} is equal to the $\,C$ term in the
right hand side of \eqref{ddt}. The second estimate follows from the
first one (here, we change the value of the constant $C$). Since $C$
does not depend on $r_0$ and $k_0\,,$ we drop the index $0$.
Further, we define
$$
i(r)=\,{Inf}_{I(r)}\,u \,, \quad  s(r)=\,{Sup}_{I(r)}\,u \,,\quad
\om(r)=\,s(r) -\,i(r)\,.
$$
By setting in \eqref{minimus2} $k=\,i(2\,r) +\,\eta\,\om(2\,r)\,,$
where $\,\eta>\,0\,,$ and by
taking into account that for  $\,x\,\in B(k,\,r)\,$ one has%
$$
0\leq\,k-\,u(x) \leq\,\eta\,\om(2\,r)\,,
$$
it follows that
$$
i(\frac{r}{2})\geq\,i(2r) +\,\eta\,\om(2r)-\,
C\,\Big\{\,\frac{1}{r^N}\,|\,B(k,\,r)\,|\,\Big\}^\frac{1}{t}\,\eta\,\om(2r)\,.
$$
Hence,
$$
\om(\frac{r}{2})\leq\,\Big\{\,1-\,\eta \Big[\,1-\,
C\,\Big(\,\frac{1}{r^N}\,|\,B(k,\,r)\,|\,\Big)^\frac{1}{t}\,\Big]\,\Big\}\,\om(2r)\,.
$$

\vspace{0.2cm}

For convenience we replace $r$ by $2r$ in the next result.
\begin{proposition}\label{propas}
Let be $\,k=\,i(4\,r) +\,\eta\,\om(4\,r)\,,$ Then
\be\label{E521}%
\om(r)\leq\,\Big\{\,1-\,\eta \, \Big[\,1-\,
C\,\Big(\,\frac{1}{r^N}\,|\,B(k,\,2r)\,|\,\Big)^\frac{1}{t}\,\Big]\,\Big\}\,\om(4r)\,.
\ee%
\end{proposition}%
\begin{remark}\label{melhas}
\rm{In reference \cite{c2} it was proved ( \cite{c2}, equation (6.21)) that%
\be\label{exnovo}%
|B(h,\,\rho)|\,(k-\,h)^t \leq \\%
c\,\Big(\,(R-\,\ro)^{-\,t} \int_{B(k)} \,(k-\,u)^t \,dx\,+
\int_{B(k)} \, |\,\nabla\,u\,|^t \,\phi^{\,t}
\,dx\,\Big)\,|B(h,\,\rho)|^\frac{t}{N}\,.%
\ee%
This estimate, together with \eqref{eq54}, shows that
\be\label{exle64bb}%
|B(h,\,\rho)|^{\,1-\,\frac{\,t}{N}}\,(k-\,h)^t \leq \,
c_1\,(R-\,\ro)^{-\,t}
\int_{B(k)} \,(k-\,u)^t \,dx\,.%
\ee%
If we appeal to this estimate (instead of appealing to the second
estimate \eqref{exle64}) we get \eqref{minimus} with the exponent
$\,\frac{\theta-1}{t}\,$ replaced by $\,\frac{1}{N\,\theta_1}\,,$
where $\,\frac{t}{N-\,t}+\,\theta_1=\,\theta_1^{\,2}\,.$}%
\end{remark}
\section{Proof of theorem \ref{teo-nopub}}\label{regcrit2}
We start this section by stating a well known potential theory
result.
\begin{lemma}\label{L62}
Let $\,\mu\,$ be a  compact supported, bounded variation measure in
$\,\R^N\,,$ and let%
\be\label{e63}%
U^{\mu}_1(x)=\,\int \, \frac{\,d\,\mu(z)}{|\,x-\,z|^{N-\,1}}%
\ee%
be the potential of order $\,1\,$ generated by $\,\mu\,.$ Then,
there is a positive constant $\,c\,$ such that
\be\label{e64}%
|\,\{\,x \in\,\R^N: \,|\,U^{\mu}_1(x)\,| \geq\,\tau\,\}\,|
\leq\,\Big(\,\frac{c\,\int
\,|\,d\,\mu\,|}{\tau}\,\Big)^{\frac{\,N}{N-\,1}}\,,%
\ee%
for each $\,\tau>\,0\,$.
\end{lemma}
For potentials of order $\,2\,$, the above result is due essentially
to E. Cartan, see \cite{cartan} lemma 4. The result is easily
extended to potentials of arbitrary order $\,\alpha\,.$ For
$\,\alpha=\,1\,,$ it claims that
$$
\textrm{cap}^{\,*}_1\{\,x \in\,\R^N: \,|\,U^{\mu}_1(x)\,|
\geq\,\tau\,\}\leq\,\frac{2^{N-\,1}\,\,\int \,|\,d\,\mu\,|}{\tau}\,,
$$
for each $\,\tau>\,0\,,$ where $\,\textrm{cap}^{\,*}_1 (E)\,$
denotes the internal capacity of order $\,1\,$ of the set $\,E\,$.
Equation \eqref{e64} follows by appealing to the classical estimate
$$
|\,E\,|\leq\, c(N)\,(\,\textrm{cap}^{\,*}_1
(E)\,)^{\frac{\,N}{N-\,1}}\,.
$$
Next we prove the following result.

\begin{lemma}\label{L63}
Let be $\,0\leq\,h<\,k\leq\,m\,,$ and $\,0<\,r<\,\frac{\ro_0}{2}\,.$
Then \be\label{e65}%
\begin{array}{ll}
|\,B(h,\,2\,r)\,|^{\frac{\,t\,(N-\,1)}{N\,(\,t-\,1)}}\leq\,c\,[\,(k-\,h)\,\sigma(2\,r)\,]^{\,-
\frac{\,t}{\,t-\,1}}\,(\,|\,B(k,\,2\,r)\,|-\,|\,B(h,\,2\,r)\,|\,)\\
\\
\Big(\,(\,2\,r\,)^{\,-t} \,\int_{B(k,\,4\,r)} \, |\,u-\,k\,|^t \,dx
\,\Big)^{\frac{1}{t-\,1}}\,.%
\ea\ee
\end{lemma}
\begin{proof}
Set
$$%
v=\,\left\{%
\begin{array}{ll}%
k-\,h\ \quad \textrm{if} \quad u\leq\,h\,,\\
k-\,u \quad \textrm{if} \quad h \leq\,u\leq\,k\,,\\
0 \quad \textrm{if} \quad k\leq\,u\,,\\
\ea%
\right.%
$$%
and%
\be%
\mu(z)=\,\left\{%
\begin{array}{ll}%
|\,\na\,v(z)\,| \quad \textrm{on} \quad I(0,\,2\,r)\,,\\
\\
0 \quad \textrm{on} \quad (\complement \,I)(0,\,2\,r)\,.
\ea%
\right.%
\ee%
Since $\,v\,$ vanishes on $\,E_{2\,r}\,$, from assumption
\ref{assumpsigma} it follows $\,|\,v(x)\,|
\leq\,c\,\sigma(2\,r)^{\,-\,1}\,\,U^{\mu}_1(x)\,$ on $\,I(2\,r)\,.$
Hence, by lemma \ref{L62}, we show that
\be\label{e66}\begin{array}{ll}%
|\,\{\,x\in\,I(2\,r)\,:\,|\,v(x)\,|\geq\,\tau\,\}\,|
\leq\,c\,\Big(\,(\sigma(2\,r)\,\tau\,)^{\,-\,1}\,\int_{I(\,2\,r)}
 \,|\,\na\,v(z)\,| \, dz \,\Big)^{\frac{\,N}{N-\,1}}\,,%
\ea\ee%
for each $\,\tau>\,0\,.$ Let be $\,\tau=\,k-\,h-\,\ep\,$, where
$\,\ep>\,0\,.$\par%
By appealing to the definition of $\,v\,$ we prove that
$$
\begin{array}{ll}%
|\,B(h,\,2\,r)\,| \leq \,|\,\{\,x \in\,I(2\,r): v(x)
\geq\,\tau\,\}\,|
\\
\leq\,c\,\Big(\,\big[\sigma(2\,r)\,(k-\,h-\,\ep)\,\big]^{\,-\,1}\,\int_{B(k,\,2\,r)-\,B(h,\,2\,r)}
 \, |\,\na\,v(z)\,|  \, dz \,\Big)^{\frac{\,N}{N-\,1}}\,.%
\ea
$$
Further, by letting  $\,\ep \rightarrow 0\,$ in the last equation,
and by appealing to H\H older's inequality, we obtain the estimate
\be\label{e67}
\begin{array}{ll}%
|\,B(h,\,2\,r)\,|^{\frac{N-\,1}{\,N}}
\leq\,c\,\Big(\,\big[\sigma(2\,r)\,(k-\,h)\,\big]^{\,-\,1}\,
\Big(\,\int_{B(k,\,2\,r)} \, |\,\na\,u\,|^t \, dx \,\Big)^{\frac{\,1}{\,t}}\cdot\\
\\
\,\big(\,|\,B(k,\,2\,r)\,|-\,|\,B(h,\,2\,r)\,|\,\big)^{\frac{\,t-\,1}{\,t}}\,\Big)\,.%
\ea\ee%
Finally, by raising both terms of the last equation to the power
$\,\frac{\,t}{\,t-\,1}\,,$ and by appealing to theorem \ref{T51}
(with $\,\ro=\,2\,r\,,$ and $\,R=\,4\,r\,$) the thesis follows.
\end{proof}
\begin{theorem}\label{T61}
Let be $\,0<\,r<\,4^{\,-\,1}\ro_0\,.$  There is a constant
$\,C_1\,$, which depends at most on $\,a,\,p_0,\,d,\,t\,,$ and
$\,N\,,$ such that if $\,n_0=\,n_0(r) \,$ satisfies \eqref{enezero}
below, then%
\be\label{e68}%
\om(r)\leq\,(\,1-\,2^{\,-\,1}\,\eta_{\,n_0}\,)\,\om(4\,r)\,,%
\ee%
where
$$
\eta_{\,n_0}=\,2^{\,-(\,n_0+\,1)}\,.
$$
\end{theorem}
\begin{proof}
Let be $\,l=\,i(4r)\,,$ $\,\om=\,\om(4r)\,,$ and set, for each no-negative integer $\,j\,,$%
\be\label{E520}%
\left\{%
\begin{array}{ll}
\eta_j=\,2^{\,-\,(j+\,1)}\,,\\
\\
k_j=\,i(4r)+\,\eta_j \,\om(4r)\,,
\ea%
\right.%
\ee%
and $\,b_j=\,|\,B(k_j,\,2\,r)\,|\,.$ By lemma \ref{L63} with
$\,k=\,k_j \,$ and  $\,h=\,k_{\,j+\,1} \,,$ we obtain
$$
\begin{array}{ll}%
{b_{j+1}}^{\frac{t(\,N-\,1)}{\,N(t-1)}}\leq\,c\,\big[\,2^{\,-(\,j+\,2)}\,\om\,\sigma(2r)
\,\big]^{-\,\frac{t}{\,t-\,1}}\,(\,b_{j}-\,b_{j+1}\,)\cdot \\
\\
\big[\,(2r)^{\,-t}\,V_N\,(4\,r)^N\,(2^{\,-(\,j+\,1)}\,\om)^t
\,\big]^{\frac{\,1}{\,t-1}}\,.
\ea%
$$
Straightforward calculations show that%
\be\label{e610}%
\begin{array}{ll}%
{b_{j+1}}^{\frac{t(\,N-\,1)}{\,N(t-1)}}\leq\,c\,r^{\frac{N-\,t}{\,t-\,1}}\,\sigma(2r)^{\,-
\frac{\,t}{\,t-1}}\,(\,b_{j}-\,b_{j+1}\,)\,,%
\ea\ee%
where, for convenience, the value of the constant $c$ may change
from equation to equation (clearly, it depends only on fixed quantities like $N$, $t$, etc.).\par%
Denote by $\,n_0=\,n_0(r)\,$ an arbitrary positive integer, to be
fixed later on. From \eqref{e610} it follows that
$$
{b_{n_0}}^{\frac{t(\,N-\,1)}{\,N(t-1)}}\leq\,{b_{j+1}}^{\frac{t(\,N-\,1)}{\,N(t-1)}}\leq\,
c\,r^{\frac{N-\,t}{\,t-\,1}}\,\sigma(2r)^{\,-
\frac{\,t}{\,t-1}}\,(\,b_{j}-\,b_{j+1}\,)\,,
$$
for each $j\,$, $\,0\leq\,j\leq\,n_0 -\,1\,.$ Consequently,
$$
\begin{array}{ll}%
n_0\,{b_{n_0}}^{\frac{t(\,N-\,1)}{\,N(t-1)}}\leq\,
c\,r^{\frac{N-\,t}{\,t-\,1}}\,\sigma(2r)^{\,- \frac{\,t}{\,t-1}}\,
\sum_{j=\,0}^{n_0-\,1} \, (\,b_{j}-\,b_{j+1}\,)\\%
\\
\qquad \qquad \qquad \leq\,c_0\,\sigma(2r)^{\,- \frac{\,t}{\,t-1}}\,
(2r)^{\frac{\,t(N-\,1)}{\,t-1}}\,.%
\ea%
$$%
Hence,
\be\label{e612}%
\Big(\,{\frac{b_{n_0}}{(2\,r)^N}}\Big)^{\frac{1}{t}}
\leq\,C\,{n_0}^{-\frac{N(t-1)}{t^2\,(N-1)}}
\,\sigma(2r)^{\,- \frac{\,N}{\,t(N-1)}}\,.%
\ee%
On the other hand, from \eqref{E521}, one has
\be\label{E521bis}%
\om(r)\leq\,\Big\{\,1-\,2^{-(n_0+1)} \, \Big[\,1-\,
C\,\Big(\,\frac{b_{n_0}}{r^N}\,\,\Big)^\frac{1}{t}\,\Big]\,\Big\}\,\om(4r)\,.
\ee%
Finally, from \eqref{e612} and \eqref{E521bis},
\be\label{E521tis}%
\om(r)\leq\,\Big\{\,1-\,2^{-(n_0+1)} \, \Big[\,1-\,
\,C_0\,{n_0}^{-\frac{N(t-1)}{t^2\,(N-1)}} \,\sigma(2r)^{\,-
\frac{\,N}{\,t(N-1)}}\,\Big]\,\Big\}\,\om(4r)\,.
\ee%
Next, we want to single out an index $n_0=\,n_0(r)$ such that the
expression under square brackets is less or equal to $\frac12\,,$
for each positive (small) radius $r$. This leads to
\be\label{enezero}%
n_0(r)\geq\,C_1\,\sigma(2r)^{\,- \frac{\,t}{\,t-1}}\,,%
\ee%
where $C_1$ is a constant which depends at most on
$\,a,\,p_0,\,d,\,t\,,$ and $\,N\,$. In the sequel we denote by
$n_0(r)$ the smallest integer for which \eqref{enezero} holds. Hence
\be\label{enezero3}%
C_1\,\sigma(2r)^{\,- \frac{\,t}{\,t-1}}\leq\,n_0(r)<\,1+\,C_1\,\sigma(2r)^{\,- \frac{\,t}{\,t-1}}\,.%
\ee%
\end{proof}
\begin{lemma}\label{L64}
Let $\,C_1\,$ be the constant in equation \eqref{enezero}. Then,
\begin{equation}\label{densasr2}
\big[\,\sigma(r)\,\big]^{\frac{t}{t-\,1}}\geq\,C_1\,(\log
2)\,(\log\,\log\,(r^{-\,1})\,)^{-1}\,,
\end{equation}
for each positive $\,r\,,$ in a arbitrarily
small neighborhood of zero (clearly, $\,r<\,1\,$ is assumed), then%
\be\label{e613}%
\lim_{r \rightarrow \,0} \om(r)=\,0\,.%
\ee%
In particular, the boundary point $\,y\,$ is regular.
\end{lemma}
\begin{proof}
Fix a positive $\,r_0\,$ such that
\be\label{e614}%
\om(r)\leq\,(1-\,4^{\,-1}\,\eta_{\,n_0}\,)\,\om(4\,r)\,, \quad
\forall \, r<\,r_0\,.%
\ee%
This choice is possible, by \eqref{e68}. Further, define,
for each no-negative index $\,i\,$,%
\be\label{e615}%
r_i=\,4^{\,-\,i}\,r_0\,.%
\ee%
Furthermore, set $\,n_0(i)=\,n_0(r_i)\,.$ From \eqref{e614} it
follows that
$\,\om(r_i)\leq\,(1-\,4^{\,-1}\,\eta_{\,n_0(i)}\,)\,\om(r_{\,i-1})\,,$
for each $\,i\geq\,1\,,$ so%
\be\label{e616}
\om(r_i)\leq\,\prod_{k=\,1}^{i}(1-\,4^{\,-1}\,\eta_{\,n_0(k)}\,)\,\om(r_0)\,.%
\ee%
From \eqref{enezero3}, and \eqref{densasr2}, it follows that
$$
n_0(r) <\,1+\,(\log
2\,)^{-1}\,\log\big(\,\log(\,2\,r)^{\,-\,1}\,\big)\,.
$$
Hence,%
$$
2^{\,n_0(k) +\,1} \leq\,4\,
e^{\log\big(\,\log(\,2\,r)^{\,-\,1}\,\big)}=\,4\,\log(\,2\,r)^{\,-\,1}\,,
$$
where $\,r=\,r_k\,.$ It follows that%
\be\label{e618}%
\eta_{n_0(k)}\geq\,4^{\,-1}\,\big(\,\log\,(2\,r_k)^{-1}\,\big)^{\,-1}\,,
\quad \forall \,k \geq\,1\,.%
\ee%
Further, by appealing to \eqref{e615}, one gets
\be\label{e618}%
\eta_{n_0(k)}\geq\,\frac{1}{4\,\big(k\,\log\,4-\,\log\,(2\,r_0)\,\big)}\,.%
\ee%
Since $\,\log(1-\,x) \leq\,-\,x\,$ we get
$$
\log(\,1-\,4^{\,-1}\,\eta_{n_0(k)}\,)\leq\,\frac{-1}{4^2\,\big(k\,\log\,4-\,\log\,(2\,r_0)\,\big)}\,.%
$$
So,
$$
\sum_{k=\,1}^{+\,\infty}\,
\log\,(1-\,4^{\,-1}\,\eta_{\,n_0(k)}\,)\,=\,-\,\infty\,.
$$
Hence%
\be\label{e619}
\prod_{k=\,1}^{+\,\infty}(1-\,4^{\,-1}\,\eta_{\,n_0(k)}\,)=\,0\,.%
\ee%
Equation \eqref{e613} follows from \eqref{e616} and \eqref{e619}\,.
\end{proof}
\begin{remark}
\rm{ In the more general situation \eqref{zerotres},
\eqref{zeroquatro}, \eqref{zerocinco}, one has to appeal to
\eqref{doistres}. In this case \eqref{e68} is replaced by
\be\label{e68b}%
\om(r)\leq\,(\,1-\,2^{\,-\,1}\,)\,\om(4\,r)+\,(\,c+\,\eta_{\,n_0}^{-1})\,r\,.%
\ee%
So, in the proof of lemma \ref{L64}, one has to consider also the
event of the non existence of a positive $\,r_0\,$ for which
\eqref{e614} holds.}%
\end{remark}
Proof of Theorem \ref{teo-nopub}:
\begin{proof}
From lemma \ref{L64} we conclude that the capacitary potentials
$\,u_{\ro,\,m}(x)\,$ are continuous at point $\,y$. Since
$\,u_{\ro,\,m}=\,m\,$ on $\,E_{\ro_0}\,,$ and
$\,|\,E_{\ro}\,|>\,0\,$ for each positive $\,\ro\,,$ it must be
$\,u_{\ro,\,m}(y)=\,m\,.$  The continuity of the potentials
$\,u_{\ro,\,- m}(x)\,$ at $\,y\,$, and $\,u_{\ro,\,-
m}(y)=\,-\,m\,,$ are proved in a totally similar way or,
alternatively, by appealing to the  remark \ref{rem-1.1}.\par%
Finally, the regularity of the boundary point $\,y\,$ follows from
theorem \ref{teo-B}.
\end{proof}

\end{document}